\documentclass[12pt,reqno]{article}
\usepackage[body={6.5in,9.0in},left=1in,top=1in,centering]{geometry}
\usepackage{graphicx}
\usepackage{amssymb}
\usepackage{epstopdf}
\usepackage{amssymb,amsmath,amsthm,amsfonts,color,mathrsfs}
\usepackage[colorlinks=true,urlcolor=blue,linkcolor=blue,citecolor=blue,pdfstartview=FitH]{hyperref}
\usepackage{xcolor}
\usepackage{authblk,pstricks}
\usepackage{tikz}
\usetikzlibrary{shapes,arrows,positioning,fit,backgrounds,calc}
\usepackage{pgfplots}
\usepackage{subcaption}
\pgfplotsset{compat=1.16}
\usepackage[sort,comma]{natbib} 
\usepackage[title]{appendix}
\usepackage{mathrsfs}
\newcmykcolor{orange}{0 .61 0.87 0}
\newcmykcolor{navyblue}{0.94 0.74 0 0}
\newcmykcolor{peach}{0 0.50 0.70 0}
\DeclareMathOperator*{\argmax}{arg\,max} 
 
\newtheorem{thm}{Theorem}[section]

\newtheorem{cor}[thm]{Corollary}

\newtheorem{prop}[thm]{Proposition}
\newtheorem{lem}[thm]{Lemma}
\theoremstyle{definition}
\newtheorem{defn}[thm]{Definition}

\newtheorem{cnd}[thm]{Condition}

\newtheorem{rem}[thm]{Remark}

\newcommand{\lemref}[1]{Lemma~{\rm \ref{#1}}}

\newcommand{\cndref}[1]{Condition~{\rm \ref{#1}}}
\newcommand{\propref}[1]{Proposition~{\rm \ref{#1}}}
\newcommand{\defref}[1]{Definition~{\rm \ref{#1}}}
\newcommand{\remref}[1]{Remark~{\rm \ref{#1}}}

\newcommand{\sectref}[1]{Section~{\rm \ref{#1}}}

\makeatletter \@addtoreset{equation}{section}

\allowdisplaybreaks

\def\R{\ensuremath {\mathbb R}}
\newcommand{\EE}{\mathbb E}

\newcommand{\NN}{\mathbb N}
\newcommand{\RR}{\mathbb R}
\def\P{\ensuremath{\mathbb P}}
\newcommand{\F}{\mathcal F}
\newcommand{\I}{{\mathcal I}}
\newcommand{\A}{{\mathcal A}}
\newcommand{\AS}{\mathcal A_{\mathrm{Sing}}}
\newcommand{\AF}{\mathcal A_{\mathrm F}}
\newcommand{\AI}{\mathcal A_{\mathrm I}}
\newcommand{\B}{{\mathcal B}}
\newcommand{\cR}{\mathcal R}

\newcommand{\e}{\varepsilon}\newcommand{\eps}{\epsilon}

\newcommand{\id}{\mathfrak i}
  
\newcommand{\wdh}{\widehat}
\newcommand{\wdt}{\widetilde}

\newcommand{\lan}{\langle} \newcommand{\ran}{\rangle}

\renewcommand{\hat}{\widehat}

\newcommand{\E}{{\cal E}}

\newcommand{\comment}[1]{} 

\title{Long-Term Average Impulse and Singular Control of a Growth Model with Two Revenue Sources\thanks{This research was supported by the Simons Foundation under grant number 8035009.}}
\author[1]{K.L. Helmes}
\author[2]{R.H. Stockbridge}
\author[2]{C. Zhu}
\affil[1]{\small Institute for Operations Research, Humboldt University of Berlin, Spandauer Str. 1, 10178, Berlin, Germany, {\tt helmes@wiwi.hu-berlin.de}}
\affil[2]{Department of Mathematical Sciences,   University of Wisconsin-Milwaukee,   Milwaukee, WI 53201,   USA,   {\tt stockbri@uwm.edu}, {\tt zhu@uwm.edu}}
 
\date{\today}

\begin{document} 
\maketitle 
\begin{abstract}
This paper analyzes and explicitly solves a class of long-term average impulse control problems and a related class of singular control problems.  The underlying process is a general one-dimensional diffusion with appropriate boundary behavior.  The model is motivated by applications such as the optimal long-term management of renewable resources and financial portfolio management.  A large class of admissible policies is identified over which the agent seeks to maximize her long-term average reward, consisting of a running reward and income from either discrete impulses or singular actions.    In addition to the long-term average objective, we will briefly consider the long-term expected total reward functional and its relation to overtaking optimality.  Sensitivity analysis with regard to the parameters of the impulse control model are performed. Key connections between the impulse and singular control problems are displayed. 
\end{abstract}
%
\maketitle
\section{Introduction} \label{sect-intro}
 
This paper considers and explicitly solves both a long-term average stochastic impulse control problem and a related singular control problem, with particular emphasis on impulse control.  Our motivation stems from two sources.  The first is the applications in natural resource management, specifically in the context of optimal and sustainable harvesting strategies.  The second is mathematical in nature.  It concerns  (a) the important but subtle interplay between two revenue streams, the incomes from a running reward and from intervention decisions, (b) the exploration of a direct approach to analyze such control problems over a large class of admissible policies using renewal theory and the renewal reward theorem, and (c) ways in which the impulse and singular problems are related.   An additional motivation arises from the connection of the notion of optimality for an impulse problem with accumulated rewards to the concept of overtaking policies and their relation to an optimal policy for the long-term average problem.  

Let us now introduce the underlying uncontrolled process.  In the absence of controls, the dynamics of the state process -- which may describe the evolution of some renewable natural resource  -- is modeled by  a one-dimensional diffusion process  on an interval $\I \subset \RR$  
\begin{equation} \label{e:X0}
d X_{0}(t) = \mu(X_{0}(t))\, dt +\sigma(X_{0}(t))\, dW(t), \ \ X_{0}(0) = x_{0},
\end{equation}  
where $x_{0}\in \I$ is an arbitrary but fixed point throughout the paper,  $W$ is a one-dimensional standard Brownian motion, and the drift and diffusion are given by the functions $\mu$ and $\sigma$, respectively.  The diffusion process is assumed to have certain boundary behaviors,   which are consistent with growth models arising from important applications; see \cndref{diff-cnd} for details.

For the impulse control problems under consideration, the decision maker wants to specify when and by how much the state of the process should be reduced to achieve economic benefits.  Her strategy is modeled by a policy $R : =\{(\tau_{k}, Y_{k}), k =1,2,\dots\}$ such that for each $k\in \NN$,  $\tau_{k}$ is the time of the $k$th intervention and $Y_{k}$ is the size of the intervention.  The resulting  controlled process $X^{R}$ satisfies  
\begin{equation} \label{e:X} 
X^{R}(t) =x_{0} + \int_{0}^{t} \mu(X^{R}(s))\, ds + \int_{0}^{t} \sigma(X^{R}(s))\, dW(s) - \sum_{k=1}^{\infty}I_{\{\tau_{k} \le t\}} Y_{k},  \qquad t\ge 0,
\end{equation} 
where we set $\tau_{0} =0$ and $ X^{R}(0-) = x_{0} \in \I$.  In addition, since this paper is concerned with long-term average problems,  we restrict ourselves to policies  with  $\lim_{k\to\infty}\tau_{k} = \infty$ a.s.

As well as earning income at discrete times when the state is reduced for economic gain, the manager continually receives a subsidy based on the state of the process.  Let $c$ be a running reward function and $\gamma > 0$ a scaling parameter.  The function $\gamma c > 0$ gives the net reward after the associated running costs are taken into account.  Let $p_1>0$ be the gross price received per unit when a reduction in the state process is initiated.  We assume that the costs for instantaneous reductions consist of both a fixed amount $K > 0$ for each action and a proportional cost with parameter $q_1 < p_1$.  Set $p = p_1 - q_1 > 0$ to be the net price per unit of reduction due to an intervention.  Thus the reward functional for a policy $R$ is the expected long-term average profit:
\begin{equation}\label{e:reward-fn} 
J(R; p, K, \gamma) := J(R) := \liminf_{t\to\infty}  t^{-1} \EE \left[\int_{0}^{t} \gamma c(X^{R}(s))\, ds +  \sum_{k=1}^{\infty} I_{\{\tau_{k} \le t\}} (p Y_k - K)\right].
\end{equation} 
In the context of  harvesting problems, the function $\gamma c$ can represent the utility derived from maintaining desirable state values $X^{R}(s)$  at time $s$, as well as the state's contribution to the overall ecosystem's stability.  For example, the function $\gamma c$ can be used to model running carbon credits for managing large tracts of forest. 
The positive fixed cost for each intervention  in \eqref{e:reward-fn} implies that the optimal policy involves discrete interventions rather than continuous adjustments, ensuring effective product management while maximizing  the overall profit rate.

Due to the presence of the nonnegative running reward rate $c$, the interplay between $c$ and the mean production/growth rate $\mu$ is one of the essential and important   features of
the model and requires careful analysis.  Although the case of a negative function $c$ is typical in applications such as  inventory control and industrial animal husbandry, a negative running or holding cost term in fact simplifies both the analysis and the characterization of optimal controls near the right boundary of the state space. Specifically, a negative $c$ prizes interventions that keep the controlled process away from the right boundary, thereby avoiding complications associated with boundary behavior of the underlying diffusion. By contrast, a nonnegative running reward function $c$ may encourage the manager to allow the process to approach the right boundary. 

Although the impulse control model considered here resembles those analyzed in our prior work (e.g., \cite{helm:18,HelmSZ-25}), a crucial difference exists: we incorporate a positive running reward rather than a negative holding cost. This distinction, particularly when combined with an attracting natural right boundary, necessitates a careful re-examination of several fundamental results.  Specifically, we identify  conditions on the model under which a strategy analogous  to an $(s, S)$-type ordering policy in the context of inventory control is optimal and adopt further conditions that guarantee the uniqueness of such a policy. The details are in Propositions \ref{prop-optimal-control} and \ref{prop-sec3-uni+existence}.  (In inventory theory, $s$ is the threshold at which orders are placed to raise the inventory to $S$; in this paper, $S$ is the threshold at which the state is lowered to the reset level $s$.  We adopt the notation $(s,S)=(w,y)$ for this type of policy, see \defref{thresholds-policy}.)

Turning to the singular control problem, as before, the underlying uncontrolled process $X_0$ satisfies \eqref{e:X0}.  A (singular) control $Z$ is a nondecreasing and right continuous process, with the resulting state process $X^Z$ satisfying 
 \begin{equation}\label{e:SDE-XZ}
 X^{Z}(t) = x_0 + \int_0^t \mu(X^{Z}(s))\, ds + \int_0^t \sigma(X^{Z}(s)) dW(s) - Z(t), \quad t\ge 0. 
 \end{equation}
The long-term average reward associated with the singular control $Z$ is  
 \begin{equation}\label{e:singular-reward}
 \wdh J(Z; p,\gamma) := \wdh J(Z) := \liminf_{t\to\infty} t^{-1} \EE_{x_0}\left[\int_{0}^{t} \gamma\, c(X^{Z}(s))\, ds + p\, Z(t)\right].
 \end{equation}

Observe that for a given impulse policy $R =\{(\tau_{k}, Y_{k}), k =1,2,\dots\}$, we may define the nondecreasing, right continuous process $Z$ by
$$Z(t) = \sum_{k=1}^\infty I_{\{ \tau_k \leq t\}} Y_k, \qquad t \geq 0,$$
in which case the two equations \eqref{e:X} and \eqref{e:SDE-XZ} describe the same controlled process $X^R = X^Z$.  Thus the collection of singular policies contains the impulse policies.

Now briefly compare the singular long-term average payoff rate \eqref{e:singular-reward} with the similar quantity \eqref{e:reward-fn} for the impulse control problem for a given impulse policy $R$.  Both problems have the same running reward.  For each $k \in \NN$, when the nondecreasing process $Z$ has a jump discontinuity at time $\tau_k$, $\Delta Z(\tau_k):= Z(\tau_k)- Z(\tau_k-) = Y_k$ so the corresponding revenue generated at time $\tau_k$ by the singular problem is $p\, Y_k$ whereas for the same intervention the impulse problem also incurs the positive fixed cost $K$ so generates the smaller revenue $p\, Y_k - K$.  Intuitively, one can regard the singular control problem \eqref{e:singular-reward} as a limiting case of the impulse control problem \eqref{e:reward-fn} when the fixed cost $K$ approaches zero.  This intuition will be made precise and rigorously established in Section \ref{sect-imp_sing}.

To avoid penury, the impulse control problem with a fixed cost $K>0$ avoids policies with infinitely many interventions in any finite time interval.  In contrast, with $K=0$,  the singular control problem allows for continuous control such as a local time process that reflects the state to prevent it from growing too large. The term $p Z(t)$ then includes the total revenue earned from these infinitesimal reductions (up to time $t$) which result in the reflection of $X^Z$.

 Since the seminal work by \cite{BensL-75} and the monograph \cite{BensL-84}, which lays the mathematical foundations of impulse control theory, stochastic impulse control has been extensively studied. Notable developments include investigations into the impulse control of one-dimensional Itô diffusions (\cite{AlvaL:08, JackZ-06}), applications in inventory management (\cite{HelmSZ-17, helm:18, HelmSZ-25, YaoCW-15, HeYZ-17}), and various financial models (\cite{EastH-88, Korn-99, Cadenillas-06}). 

Stochastic singular control has an even  longer history, dating back to the pioneering work of \cite{BatherC-67,BatherC:67}, which inspired subsequent   research  such as \cite{Alvarez,BenesSW-80,Kara-83,Weera-02,GuoP-05,Shreve-88,Weerasinghe-07,JackZ-06}, among others.  
 More recent developments include applications in  reversible investment (\cite{DeAngelisF-14}), optimal harvesting and renewing (\cite{HeniT-20,HeniTPY-19}), and dividend payment and capital injection  (\cite{JinYY-13,LindL-20}). We also refer to  \cite{ChrisMO-25,LiangLZ-25,Hynd-13,KunwXYZ-22,MenaR-13,WuChen-17} for developments on ergodic singular control problems in various settings.
 
The rest of the paper is organized as follows.  Section  \ref{sect-formulation}  gives the precise model formulations and collects the key conditions used in later sections.  In particular, a large class of admissible policies is defined by requiring a weak transversality condition to be satisfied, based on the fundamental quantities of the underlying model.  
Section~\ref{sect-preliminaries} introduces the class of $(w,y)$-impulse policies.  
Using a renewal argument for the impulse control problem, we then express the corresponding payoff \eqref{e:reward-fn} as the nonlinear function $F$ defined in \eqref{e:F_K} and derive the analytical form of the controlled process's stationary density in \eqref{e:nu_density-wy}.  Of particular significance, Section~\ref{sect-preliminaries} also analyzes the asymptotic behavior of various components of the function $F$, which will be used in the maximization of $F$, a fractional optimization problem.  

The long-term average optimal impulse control problem is addressed in Section~\ref{sect:classical impulse}. Conditions are given for the existence and uniqueness of an optimal $(w,y)$-policy.  Furthermore, it is shown that the ``impulse reward potential'' function $G$ (defined up to a constant in \eqref{e:G_p}) along with the maximal value $F^*$  is a solution of the associated system of quasi-variational inequalities \eqref{e:qvi}. This, in turn, establishes the optimality of the corresponding $(w,y)$-policy.   The potential $G$ is then used to define the natural solution $\wdt G$ of the system \eqref{e:qvi}.  A brief digression examines the growth of the long-term expected total reward.  A specific shift of $\blue \wdt G$ is interpreted as the value function of an infinite-horizon growth rate optimization problem.  The section concludes by relating the optimal $(w,y)$-policy to the concept of an optimal overtaking policy.  

Using the framework established in \sectref{sect:classical impulse}, the singular control problem is defined and quickly dispensed with in \sectref{sect:singular}.  \sectref{sect:sensitivity} examines the sensitivity of the solutions for the impulse control problem to changes in the model parameters $(p, K, \gamma)$. Section \ref{sect-imp_sing} discusses the intrinsic relationship between the impulse control and singular control problems. In particular, the first-order optimality conditions for the impulse problem lead to the asymptotic result that the singular control solution emerges as the limit of the impulse solutions as the fixed cost $K$ converges to $0$.   Surprisingly, the reduction in the long-term average reward rate due to the positive fixed cost $K$ is quantified by a comparison of the values corresponding to two reflection policies; see \remref{quantifying-fixed-cost}.  The paper concludes with a summary of findings and final remarks in Section \ref{sect-conclusion}.  Some technical but subsidiary proofs of Sections \ref{sect-formulation}, \ref{sect-preliminaries}, and \ref{sect:classical impulse} are arranged in Appendix \ref{Appen-proofs-sect2}.

Throughout the paper, we use the notation that $\lan f,\pi \ran : = \int f d\pi$ if $f$ is a function and $\pi$ is a measure, as long as the integral $ \int f d\pi$ is well-defined.  The indicator function of a set $A$ is denoted  by $I_{A}$.

\section{Formulation}\label{sect-formulation} 

In this section, we establish the models under consideration and collect some key conditions that will be used in later sections of the paper.
\medskip

\noindent
{\bf \em Dynamics.}\/ 
Let $\I: = (a, b) \subset \RR$ with $a> -\infty$ and $b\le \infty$. In the absence of interventions, the process $X_{0}$ satisfies \eqref{e:X0} and is a regular diffusion with state space $\I$.   The measurable functions $\mu$ and $\sigma $ are assumed to be such that a unique non-explosive weak solution to  \eqref{e:X0} exists; we refer to Section 5.5 of \cite{Karatzas-S} for details.   For simplicity, we assume $(\Omega, \F, \{\F_{t}\}, \P)$ is  a filtered probability space with an $\{\F_{t}\}$-adapted Brownian motion $W$ and on which $X_0$ is defined, as well as each controlled process. In addition, we assume that  $\sigma^{2}(x) > 0$ for all $x \in \I$.  We closely follow the notation and terminology on boundary classifications of one-dimensional diffusions in Chapter 15 of \cite{KarlinT81}.  The following standing assumption is imposed on the model throughout the paper:

\begin{cnd} \label{diff-cnd}
\begin{itemize}
\item[(a)] Both the speed measure $M$ and the scale function $S$ of the process $X_0$ are absolutely continuous with respect to Lebesgue measure. The scale and speed densities, respectively, are given by 
\begin{equation}
\label{e:s-m}
s(x) : = \exp\left\{-\int_{x_{0}}^{x} \frac{2\mu(y)}{\sigma^{2}(y)}dy \right\}, \quad m(x) = \frac{2}{\sigma^{2}(x) s(x)}, \quad x\in (a,b),
\end{equation} where $x_{0}\in \I$ is as in \eqref{e:X0} and is an arbitrary point, which will be held fixed.   (Note we follow the convention of using the factor $2$ in the definition of $m$.)
\item[(b)]  The left boundary $a > -\infty$ is a non-attracting point and the right boundary $b \le \infty$ is a natural point. 
 Moreover,  
 \begin{equation} \label{e:M(a-y)-finite}
  M[a, y] < \infty \text{ for each }y \in \I, 
\end{equation} 
and the (hitting-time) potential function $\xi$ defined by  
\begin{align}\label{e-xi}
 \xi(x) :=  \int_{x_{0}}^{x} M[a, v] dS(v), \quad  x\in \I
\end{align}  
satisfies
\begin{equation}\label{e-sM-infty}
    \lim_{x\to b} \xi'(x) = \lim_{x\to b} s(x) M[a,x] =\infty.
  \end{equation}
\item[(c)] The function $\mu$ is continuous on $\I$ and extends continuously to the boundary points  with $|\mu(a)| < \infty$.
\end{itemize}
\end{cnd} 

 Condition~\ref{diff-cnd}(a) places restrictions on the model \eqref{e:X0} which are quite natural for harvesting problems and other applications, such as in mathematical finance.
 The assumption that $a> -\infty$ is a non-attracting point implies that it cannot be attained in finite time by the uncontrolled diffusion.  For growth models  with $a=0$, this condition excludes the possibility of extinction. Likewise,  $b\le \infty$ being a natural boundary prevents the state from exploding to $b$ in finite time.  Note that $a$ can be either an entrance point or a natural point; the state space for $X_{0}$ is respectively $\E = [a, b)$ or $\E = (a, b)$.    
 
\cndref{diff-cnd}(b,c) imposes further limitations on the model.  For instance, the assumption that $|\mu(a)| < \infty$ excludes the consideration of Bessel processes.  In addition, the finiteness condition \eqref{e:M(a-y)-finite} always holds when $a$ is an entrance boundary, and implies that the expected hitting times from $w$ to $y$ are finite whenever $a<w<y<b$.   This requirement eliminates some diffusions when $a$ is natural;  see Table~6.2 on p.~234 of \cite{KarlinT81}. 

The next proposition highlights implications of Condition~\ref{diff-cnd} and introduces the useful identity \eqref{e:1/s-identity}, which will be crucial for  both theoretical and numerical calculations.
The proof of Proposition \ref{lem-mu-at-b} is placed in Appendix~\ref{appen-pfs-sec2}  to preserve the flow of the presentation.

\begin{prop}\label{lem-mu-at-b} Assume Condition \ref{diff-cnd} holds.   Then
\begin{itemize}
  \item[(i)] $\mu(a) \geq 0$. Moreover, if  $\mu(a) > 0$ then  $a$ is an entrance point;
  \item[(ii)] for each $x \in \I$, 
\begin{equation}
\label{e:1/s-identity}
 \int_{a}^{x} \mu(u) m(u)\, du =  \int_{a}^{x} \mu(u) dM( u) = \frac1{s(x)};
\end{equation} 

  \item[(iii)] if  $M[a, b] < \infty$, then  $\langle \mu, \pi \rangle = 0$ in which $ \pi(du) = \frac{dM(u)}{M[a, b]}$;
  \item[(iv)] if $M[a, b] = \infty$,    
   then $\mu(b) = 0$.
\end{itemize}
\end{prop} 

\propref{lem-mu-at-b}(iii) and (iv) have an intuitive interpretation.  In (iii), $\pi$ is the stationary distribution for the uncontrolled process $X_0$ so this result means that the mean drift rate under $\pi$ is $0$, which accords with the notion of stationarity.  Somewhat similarly, the ``stationary'' distribution in (iv) is a unit point mass at $b$ and again the result indicates the (degenerate) mean drift is $0$.

\medskip
We now specify the class of admissible policies for the impulse control problem which, apart from the transversality condition when $a$ is a natural boundary, is quite standard.  

\begin{defn}[Impulse Admissibility]  \label{admissible-policy}
We say that $R : =\{(\tau_{k}, Y_{k}), k =1,2,\dots\}$ is an {\em admissible impulse policy}\/ if 
\begin{enumerate}
  \item[(i)] $\{\tau_{k}\}$ is an increasing sequence of $\{\F_{t}\}$-stopping times with $\lim_{k\to\infty} \tau_{k} =\infty$, 
  \item[(ii)]  for each $k\in \NN$,  $Y_{k}$ is $\F_{\tau_{k}}$-measurable with $0 < Y_k \leq X^R(\tau_k-) - a$ when $\tau_k < \infty$, where equality is only allowed when $a$ is an entrance boundary; 
  \item[(iii)] $X^R$ satisfies \eqref{e:X} and we set $\tau_{0} =0$ and $X^{R}(0-) = x_{0}\in \I$; and  
 \item[(iv)]  {\em if} $a$ is a natural boundary, either 
 \begin{enumerate}
 \item[(a)] there exists an $N \in \mathbb N$ such that $\tau_N =\infty$, which implies $\tau_k = \infty$ for all $k \geq N$ and, to completely specify the policy, we set $Y_k = 0$ for all $k \ge N$; or
 
 \item[(b)] $\tau_k < \infty$ for each $k \in \mathbb N$ and, for the function $\xi$  defined in \eqref{e-xi}, it holds that 
 \begin{equation} 
   \label{eq-xi-transversality}
    \lim_{t\to \infty} \limsup_{n\to \infty}t^{-1} \EE[\xi^{-}(X^R(t\wedge \beta_{n}) )]  = 0,
   \end{equation}
 where $\xi^-$ denotes the negative part of the  time potential $\xi$, and  for each $n\in \NN$,  $\beta_n := \inf\{t  \geq 0: X^R(t) \notin( a_n,  b_n)\}$, in which   $\{a_n\}\subset \I $ is a decreasing sequence with $a_n \to a$ and $\{b_n\}\subset \I $ is an increasing sequence with $b_n \to b$.  
\end{enumerate} 
\end{enumerate} 
 We denote by $\A_{\rm Imp}$ the set of admissible impulse strategies. 
\end{defn}


\begin{rem} \label{rem-admissible-policy} 
 An admissible impulse policy $R$ satisfying Definition \ref{admissible-policy} (i), (ii), (iii), and (iv)(a) has a finite number of  interventions or no intervention (corresponding to the case when $\tau_1 =\infty$); the latter  is  called a ``do-nothing'' policy and is denoted by $\mathfrak R$.  For convenience of later presentation, we denote by $\AF$ the set of admissible policies with  at most a finite number of interventions.   
 
 On the other hand, if $R\in \A_{\rm Imp}$ satisfies Definition \ref{admissible-policy} (i), (ii), (iii), and (iv)(b), the number of interventions is infinite; the set of such  policies is denoted by $\AI$. We have 
 $$\A_{\rm Imp} =  \AF \cup \AI \quad \text{  and }\quad  \AF \cap \AI = \emptyset.$$  
 Note that \eqref{eq-xi-transversality} is a transversality condition {\em imposed only on diffusions for which $a$ is a natural boundary.}  It is satisfied by the $(w,y)$-policies defined in \defref{thresholds-policy}; see  \lemref{lem-(wy)admissible} for details.  When $a$ is an entrance boundary, \eqref{eq-xi-transversality} is automatically satisfied because $\xi$ is bounded below; see the proof of Lemma \ref{lem-Gp-limit} for details. Finally we point out that the $\{ \beta_n\}$ sequence in \eqref{eq-xi-transversality} satisfies   $\lim_{n\to\infty} \beta_{n} =\infty$ a.s. since  $a$ is non-attracting and $b$ is natural thanks to Condition \ref{diff-cnd}.  

  We emphasize that the admissible impulse policies defined in Definition \ref{admissible-policy} are not required to be of any particular type, such as a $(w,y)$-policy or a stationary policy.  For example, while the class of $(w,y)$-policies belongs to the admissible set  $\A_{\rm Imp}$, nonstationary policies that alternate  between a finite number of such policies are also admissible.  Requiring transversality in \defref{admissible-policy}(iv)(b) for models in which $a$ is a natural boundary is a weak condition that allows for a large class of admissible policies.
  \end{rem}

\medskip

We now turn to the formulation of the rewards.
\medskip

\noindent
{\bf \em Reward Structures.}\/  A running reward is earned at rate $\gamma\, c$, in which $c$ depends on the state of the process and $\gamma > 0$ is a scale parameter.  

For the impulse control problem, the reward is proportional to the size of the intervention,  with the unit price $p = p_1-q_1 > 0$ being the net price as described in the introduction.  Each intervention also incurs a fixed cost $K> 0$. Consequently, given the parameters $p$, $K$ and $\gamma$, the long-term average reward for the product manager who adopts the strategy $R\in \A_{\rm Imp}$ is given by \eqref{e:reward-fn}.

We assume that $c$ satisfies the following condition.

\begin{cnd} \label{c-cond}  
The function $c:\I\mapsto \R_{+}$ is continuous,     increasing, and extends continuously at the endpoints, with $0 \le c(a) < c(b) < \infty$.  
\end{cnd}

{  
We now introduce an essential condition that connects the mean growth rate $\mu$ to the subsidy function $c$.   It is the interplay between the subsidy and growth rates that gives the uniqueness result in \propref{prop-sec3-uni+existence}.  For the given parameters $\gamma$ and $p$, define the reward rate function by
\begin{equation} \label{e:r_p}
r(x; \gamma, p) = r(x) := \gamma c(x) + p \mu(x), \qquad x\in \I.
\end{equation}
The summand $\gamma c(x)$ gives the subsidy rate when the process is in state $x$ while $p \mu(x)$ gives the mean growth rate of the value of the product as the process traverses $x$ (cf. \eqref{e:mu-tauy-mean}).

\begin{cnd} \label{3.9-suff-cnd}
There exist $\wdh x, \wdt x\in \I$ with $\wdh x \le \wdt x$ so that the function $r$ of \eqref{e:r_p} is strictly increasing on $(a, \wdh x)$, constant on $(\wdh x, \wdt x)$,   and strictly   decreasing on $(\wdt x, b)$.
\end{cnd}

\begin{rem}
With a view towards a mean field game problem in which the price depends on the mean field interactions and is not a fixed parameter, it is shown in \cite{HelmSZ:25} that a sufficient condition on $\mu$ and $c$ such that this condition is satisfied for every pair $(\gamma, p)$ is that there exists some $\hat x_{\mu,c} \in \I$ so that $\mu$ is strictly increasing on $(a, \hat x_{\mu,c})$ and the functions $c$ and $\mu$ are concave on $(\hat x_{\mu,c}, b)$.
\end{rem}
}

 It is also worth pointing out that the long-term average reward for every $R\in  \AF$ is equal to that of the do-nothing policy  $\mathfrak R$, which will be shown to be equal to ${   \gamma \bar c(b)}$ in \eqref{e:reward-zero-policy}. 

To have a meaningful impulse problem, we require that the parameters $p$, $K$ and $\gamma$ be such that it is always beneficial to intervene with an impulse as compared to the do-nothing policy; \cndref{cond-interior-max}  captures this requirement in a functional form.

Now define the operators $A$ and $B$ that will be used often in the sequel.  The generator of the process $X^{R}$ between jumps (corresponding to the uncontrolled diffusion process $X_{0}$) is 
\begin{equation} \label{generator}
A f:= \frac12 \sigma^{2} f'' + \mu f' = \frac12 \frac{d}{dM}(\frac{df}{dS}),
\end{equation}  
where $f \in C^{2}(\I)$.   Define the set  $\cR: = \{(w, y) \in \E \times \E: w < y \}$. 
For any function $f: \E  \mapsto \R$ and $(w, y) \in \cR$, \begin{equation} \label{e:Bf-defn}
  Bf(w, y): = f(y) - f(w).
\end{equation} The jump operator $B$ captures the effect of an instantaneous impulse. 

We finish this section with some elementary but important observations concerning the time potential $\xi$ defined in \eqref{e-xi} and its   (running reward) companion  
\begin{align}
\label{e:g-defn}
g(x)&: = \int_{x_{0}}^{x}\int_{a}^{v} c(u)\, dM(u)\, dS(v), \quad  x\in \I, 
\end{align} 
which is well-defined and finite due to Conditions  \ref{diff-cnd}(b) and  \ref{c-cond}.  Both $\xi$ and $g$ are $0$ at $x_0$, negative for $x < x_0$ and positive for $x> x_0$.    The functions $\xi$ and $g$ are twice continuously differentiable on $\I$ with  
\begin{align}
\label{e:xi-derivatives}\xi'(x) & = s(x) M[a,x],   && \xi''(x) = -\frac{2\mu(x)}{\sigma^{2}(x) } \xi'(x) + s(x) m(x), \\ 
\label{e:g-derivatives}g'(x) & = s(x) \int_{a}^{x} c(u)\, dM(u),   && g''(x) =  -\frac{2\mu(x)}{\sigma^{2}(x) } g'(x) +s(x) m(x)c(x).
\end{align} Using these derivatives, we can immediately see that  
\begin{align}\label{e:AxiAg} 
A\xi(x) =1,\ \text{ and } \  Ag(x) = c(x), \quad \forall x\in \I. 
\end{align}   
Alternatively, the  equations of \eqref{e:AxiAg} can be derived from  the second representation of the generator $A$ in \eqref{generator}.   We note also that for the identity function $\id(x) = x$, it is immediate that $A\id(x) = \mu(x)$ for all $x \in \I$. Morever,   by \eqref{e:1/s-identity},     $\id$  
admits  a double integral representation  $\id(x) =  \id(x_0) + \int_{x_{0}}^{x}\int_{a}^{v} \mu(u)\, dM(u)\, dS(v)$ for any $x\in \I$.  

Finally, the following proposition provides  probabilistic representations for the differences of the functions of $\xi$, $g$ and $\id$. Its proof, being similar to that of Proposition~2.6 of \cite{HelmSZ-17}, is omitted here for brevity.    
\begin{prop}\label{prop-tau-y-mean} Assume Conditions \ref{diff-cnd} and \ref{c-cond} hold and let  $a < w < y < b$. Denote by $\tau_{y}:=\inf\{ t > 0: X_{0}(t) = y\}$ the first passage time to $y\in\I$ of the process $X_{0}$ of \eqref{e:X0} with initial state $x_0=w$.  Then  
\begin{align}  \label{e:tau_x-b}
\EE_{w}&[\tau_{y}]  = \int_{w}^{y} S[u,y]\, dM(u) + S[w,y] M[a,w] = \xi(y)  -\xi(w) = B\xi(w,y), \\   \label{e:c-tauy-mean}
\EE_{w}& \left[\int_{0}^{\tau_{y}} c(X_{0}(s)) ds \right]  = \int_{w}^{y} c(u) S[u, y]\, dM(u) + S[w, y]\int_{a}^{w} c(u)\, dM(u) = Bg(w,y),
\end{align}   
 and  
\begin{equation} \label{e:mu-tauy-mean}
\EE_{w}\left[\int_{0}^{\tau_{y}} \mu(X_{0}(s)) ds \right]= \int_{w}^{y} \mu(u) S[u, y]\, dM(u) + S[w, y]\int_{a}^{w} \mu(u)\, dM(u) = B\id(w,y).
\end{equation}   
\end{prop} 
The difference \eqref{e:tau_x-b} gives the expected passage time of $X_0$ from $w$ to $y$, whereas \eqref{e:c-tauy-mean} and \eqref{e:mu-tauy-mean}  are the expected running reward and the expected growth during this passage time.  Based on \eqref{e:tau_x-b}, \eqref{e:c-tauy-mean} and \eqref{e:mu-tauy-mean}, we call $\xi$ the time potential, $g$ the running reward potential and $\id$ the drift potential.

Furthermore, from  \eqref{e:tau_x-b} and \eqref{e:xi-derivatives}, the limit  \eqref{e-sM-infty}  implies that the rate of increase of the expected passage time of $X_0$ from $w$ to $y$ diverges to infinity as $y$ approaches $b$.  Intuitively, this indicates that the process slows down as it approaches a finite right boundary.

\section{$(w,y)$-Impulse Policies} 
\label{sect-preliminaries}
We introduce $(w,y)$-impulse policies in \defref{thresholds-policy} and establish in \lemref{lem-(wy)admissible} that these policies are admissible in the sense of \defref{admissible-policy}. Moreover, by employing a renewal argument, we can derive the long-term average reward and establish preliminary results for these functions. We also characterize the product supply rate and determine asymptotic relations involving the  functions $\xi$ and $g$ as the thresholds $w$ and $y$ converge to the boundaries of the state space.  

\begin{defn}[$(w,y)$-Policies] \label{thresholds-policy}
 Let $(w,y) \in \cR$ and set $\tau_0 = 0$ and $X^{(w,y)}(0-)=x_0$.  Define the {\em $(w,y)$-impulse policy}\/ $R^{(w,y)}$, with corresponding state process $X^{(w,y)}$, such that for $k \in \NN$,
$$ \tau_k = \inf\{t > \tau_{k-1}: X^{(w,y)}(t-) \geq y\} \qquad \text{and} \qquad Y_k = X^{(w,y)}(\tau_k-) - w.$$ 
The   definition of $\tau_k$ must be slightly modified when $k = 1$ to be $\tau_1 = \inf\{ t\ge 0: X(t-)\ge y\}$ to allow for the first jump to occur at time 0 when $x_0\ge y$.
\end{defn}

Under this policy, the impulse controlled process $X^{(w,y)}$ immediately resets to the level $w$ at the time it would reach (or initially exceed) the threshold $y$.  For simplicity, this type of policy will be called a $(w,y)$-policy, with the fact that it consists of impulses being understood from the notation.

Observe that all interventions, except possibly the first, have a reduction of size $y-w$ and an expected cycle length of $\xi(y) - \xi(w)$.  Using a renewal argument as in \cite{helm:18,HelmSZ-25},  the long-term average rate of interventions for the policy $R^{(w,y)}$ is
\begin{equation} \label{lta-intervention-rate}
\lim_{t\to \infty} t^{-1} \EE\left[\sum_{k=1}^\infty I_{\{\tau_{k} \le t\}}\right] = \frac{1}{\xi(y) - \xi(w)}
\end{equation}
and hence the long-term average supply rate is
\begin{equation} \label{eq-kappa-Q} 
\lim_{t\to\infty}   t^{-1} \EE\left[  \sum_{k=1}^{\infty}I_{\{\tau_{k} \le t\}} (X^{(w,y)} (\tau_{k}-) - X^{(w,y)} (\tau_{k}))\right]  = \frac{y-w}{\xi(y) - \xi(w)}  
  =: {\mathfrak z}(w,y).
\end{equation}
 Furthermore, the renewal argument shows that the long-term average reward \eqref{e:reward-fn} is given by the ratio of the expected net rewards in a cycle to the expected cycle length.  This means that when the parameters are $p$, $K$ and $\gamma$, the value is $J(R^{(w,y)};p,K,\gamma) = F(w, y;p,K,\gamma)$, where    
\begin{equation}
\label{e:F_K}
F(w, y;p,K,\gamma) = F(w,y) := \frac{p(y-w) - K + \gamma(  g(y) - g(w) )}{\xi(y) - \xi(w)}, \quad  (w,y) \in \mathcal R. 
\end{equation} 
Notice that $F$ is comprised of three terms.  The first is the price $p$ times the long-term average supply rate $\mathfrak z(w,y)$.    The second term, $\frac{K}{\xi(y)-\xi(w)}$, is a reduction by the equivalent expected rate at which the fixed cost would be continually assessed during the cycle.  The third term is the fraction $\frac{\gamma ( g(y)-g(w) )}{\xi(y)-\xi(w)}$, which is the expected running reward during a cycle divided by the expected cycle length.

Using  the elegant renewal argument presented  in Chapter 15, Section 8 of \cite{KarlinT81} or the standard analytical arguments,  one can show that the  policy $R^{(w,y)}$  induces an invariant measure having density $\nu(x; w,y)$ on $\E$, where 
\begin{equation}\label{e:nu_density-wy}
\nu(x; w,y) = \nu(x) := \begin{cases}
   \varrho\, m(x)\, S[w, y]   & \text{ if }x \le w, \\
   \varrho\, m(x)\, S[x, y]   & \text{  if } w < x \le y,\\
   0 & \text{  if }   x > y,
\end{cases}
\end{equation} 
and $\varrho= (\int_{a}^{w} m(x) S[w, y] dx + \int_{w}^{y} m(x) S[x, y] dx) ^{-1} $ is the normalizing constant.  In view of the second equality of \eqref{e:tau_x-b},  we have $  \varrho	=( \xi(y) - \xi(w)) ^{-1}$.   Furthermore, using \eqref{e:c-tauy-mean} and the definition of $\nu(x)$, we can write 
\begin{displaymath}
g(y) - g(w) =  \int_{w}^{y} c(u) S[u, y]\, dM(u) + S[w, y]\int_{a}^{w} c(u)\, dM(u) = \varrho^{ -1}\int_{a}^{b} c(x) \nu(x)\, dx.
\end{displaymath} 
Similarly, with the aid of \eqref{e:1/s-identity},  straightforward  computations  leads to \begin{align*}
y-w & = \int_{w}^{y}\left( \int_{a}^{x} \mu(u) m(u) d u\right) s(x) dx 
= \varrho^{-1} \int_{a}^{b}\mu(u) \nu(u) du.
\end{align*} Consequently, the long-term average supply rate \eqref{eq-kappa-Q} can be rewritten as
 \begin{align*}  
\mathfrak z(w,y) = \frac{y -w}{\xi(y) - \xi(w)} & = \int_{a}^{b}\mu(u) \nu(u) du, 
\end{align*}
and   the long-term average running reward rate is 
\begin{align*}  
\frac{\gamma ( g(y)-g(w) )}{\xi(y)-\xi(w)} & = \int_a^b \gamma\, c(u)\, \nu(u)\, du.  \rule{0pt}{22pt}
\end{align*}
Hence the long-term average reward rate \eqref{e:F_K}  of the policy $R^{(w,y)}$ has the representation 
$$F(w, y) =  \frac{p(y-w) - K + \gamma ( g(y) - g(w) ) }{\xi(y) - \xi(w)} = \int_{a}^{b} [r(x) - K\varrho] \nu(x)\, dx,$$ 
where the instantaneous reward rate $r$ is defined in \eqref{e:r_p}.  
This reward rate is adjusted by the rate at which the fixed cost would be expected to be assessed over the entire cycle.  Thus the integral representation of $F$ gives the stationary profit rate.  Notice also that the long-term average supply rate ${\mathfrak z}(w,y)$ is the mean drift under the stationary distribution.

 \begin{lem} \label{lem-(wy)admissible}
  Assume Condition \ref{diff-cnd} holds. Then the $(w,y)$-policy $R^{(w,y)}$ defined in \defref{thresholds-policy} is admissible in the sense of \defref{admissible-policy}.
\end{lem} 
\begin{proof} Obviously, we only need to show that   \eqref{eq-xi-transversality} holds.  To this end,  denote by $X = X^{(w, y)}$ the controlled process under the  policy $R^{(w,y)}$. Also let $\beta_n$ be as in Definition \ref{admissible-policy}.  An application of Itô's formula gives
 \begin{align*} 
 \EE[\xi(X(t\wedge \beta_n))] & = \xi(x_0) +\EE\left[ \int_{0}^{t\wedge \beta_n} A\xi(X(s)) ds + \sum_{k=1}^\infty I_{\{\tau_k \le t\wedge \beta_n \}} [\xi(X( \tau_k))- \xi(X( \tau_k-))] \right] \\
 & = \xi(x_0) +\EE[t\wedge \beta_n] + (\xi(w)-\xi(y))\EE\left[ \sum_{k=1}^\infty I_{\{\tau_k \le t\wedge \beta_n \}}  \right] \\ & \quad + [(\xi(w)-\xi(x_0))- (\xi(w)-\xi(y))]\EE[I_{\{\tau_1 =0\}}]; 
  \end{align*} 
  the last term is to account for the possibility of an intervention at time 0 in the event that $x_0>y$. Dividing both sides by $t$ and then taking the limits as $n, t\to\infty$, we have from the monotone convergence theorem and \eqref{lta-intervention-rate} that 
\begin{align*} 
 \lim_{t\to\infty}\lim_{n\to\infty} t^{-1} \EE[\xi(X(t\wedge \beta_n))] & =  1 + (\xi(w)-\xi(y))\lim_{t\to\infty} t^{-1} \EE\left[ \sum_{k=1}^\infty I_{\{\tau_k \le t  \}}  \right] \\
  & = 1 + (\xi(w)-\xi(y)) \frac1{\xi(y)-\xi(w)} =0.  
  \end{align*} 
  On the other hand, under the policy $R^{(w,y)}$, the process $X$ is bounded above by $y\vee x_0< b$. Also recall that $\xi(x)$ is negative for $x< x_0$ and positive for $x> x_0$. Therefore, we have 
\begin{align*}  
\EE[\, |\xi(X(t\wedge \beta_n))|\,]  
    & =   \EE[\xi(X(t\wedge \beta_n))I_{\{X(t\wedge \beta_n) \in[x_0, x_0\vee y] \}}]
     -\EE[\xi(X(t\wedge \beta_n))I_{\{X(t\wedge \beta_n)  < x_0 \}}] \\ 
   & =-\EE[\xi(X(t\wedge \beta_n))] + 2 \EE[\xi(X(t\wedge \beta_n))I_{\{X(t\wedge \beta_n) \in[x_0, x_0\vee y] \}}]\\ 
& \le -  \EE[\xi(X(t\wedge \beta_n))] + C, \end{align*} where $C= 2 \max_{x\in[x_0,x_0\vee y]} \xi(x) < \infty$ is a positive constant.  Thus, we have 
\begin{align*} 
0\le  \liminf_{t\to\infty}\limsup_{n\to\infty} t^{-1} \EE[\,|\xi (X(t\wedge \beta_n))|\,] & \le  \limsup_{t\to\infty}\limsup_{n\to\infty} t^{-1} \EE[\,|\xi (X(t\wedge \beta_n))|\,] \\ & \le  -  \liminf_{t\to\infty}\liminf_{n\to\infty} t^{-1} \EE[\xi(X(t\wedge \beta_n))] =  0.  \end{align*} This implies \eqref{eq-xi-transversality} and hence completes the proof.  
\end{proof}

We next present  some important observations concerning the  functions $\xi$ and $g$. These observations help  to establish the existence (Proposition \ref{prop-Fmax}) of a unique (Proposition~\ref{prop-sec3-uni+existence}) maximizer for the function $F$.   To preserve the flow of presentation, we place their proofs in Appendix~\ref{appen-pfs-sec3}.

{  The asymptotic results in Lemmas \ref{lem-limit:Bid/Bxi-a}, \ref{lem1-Bg/Bxi:a} and \ref{lem2-Bg/Bxi:b} analyze the long-term average supply and subsidy rates near the boundaries.  They are used to establish that, when $a$ is a natural boundary, the long-term average payoff expression $F$ of \eqref{e:F_K} achieves its maximal value in the interior of $\cR$ by showing that the limiting values on the boundary are smaller than some value of $F$ in the interior.  For models in which $a$ is an entrance boundary, these asymptotic results are used similarly but the maximal value of $F(w,y)$ may have $w=a$ and occur on the left boundary.  These results are given in \propref{prop-Fmax}. }

The first lemma shows that when both boundaries are natural, the long-term average supply rates converge to $0$ as the reset level $w$ approaches $a$, or as the threshold $y$ converges to $b$, or as both levels converge to the same endpoint.  These results are intuitively clear when $b < \infty$ since the process $X_{0}$ does not diffuse to either boundary in finite time so the expected cycle lengths converge to $\infty$.  The limits hold when $b = \infty$ as well and also the asymptotic ratios of the identity $\id$ to $\xi$ at both boundaries are $0$.

\begin{lem}\label{lem-limit:Bid/Bxi-a} Assume Condition  \ref{diff-cnd}   holds. 
\begin{itemize}
  \item[(i)] If $a$ is natural, then 
\begin{equation}
\label{e1:limit:Bid/Bxi-a}
\lim_{w\downarrow a} \mathfrak z(w,y_0) = \lim_{w\downarrow a} \frac{ w}{   \xi(w)}=\lim_{(w, y)\to (a, a)} \mathfrak z(w,y) = 0, 
\end{equation} where $y_{0}\in \I$ is an arbitrary fixed point. 

  \item[(ii)] Let $w_{0}\in \I$ be arbitrary. Then 
\begin{equation}
\label{e:limit:Bid/Bxi-b}
\lim_{y\to b} \mathfrak z(w_0,y)  =  \lim_{y\to b} \frac{y}{ \xi(y)  } = \lim_{(w, y)\to(b, b)} \mathfrak z(w,y)  =0. 
\end{equation} 
  \end{itemize}
\end{lem}

The previous \lemref{lem-limit:Bid/Bxi-a} is concerned with the long-term average product supply rates of $(w,y)$-policies.  The next two lemmas concentrate on the asymptotics of the long-term average running reward rates of such policies as the threshold and reset level converge to the boundaries.

\begin{lem}\label{lem1-Bg/Bxi:a} Assume Conditions \ref{diff-cnd} and \ref{c-cond} hold.  Then 
 \begin{equation}
\label{e0:g/xi-limit:a}
\lim_{(w,y)\to (a,a)} \frac{g(y) - g(w) }{\xi(y) - \xi(w)}=c(a).
\end{equation} Furthermore, if  $a$ is a natural point, then 
\begin{equation}
\label{e:g/xi-limit:a}
\lim_{w\to a} \frac{g(w) }{ \xi(w)} =  \lim_{w\to a} \frac{g(y_{0}) - g(w) }{\xi(y_{0}) - \xi(w)} = c(a),   
\end{equation} where $y_{0}$ is an arbitrary fixed number in $\I$. 
\end{lem}

\begin{lem}\label{lem2-Bg/Bxi:b}  Assume Conditions \ref{diff-cnd} and \ref{c-cond} hold.  Recall the stationary distribution $\pi$ of $X_0$ given in \propref{lem-mu-at-b}(iii) and define
\begin{equation} \label{eq:c-bar(b)}
   \bar c(b): =  \begin{cases}
c(b),     & \text{ if } M[a, b] = \infty,\\  
\lan c,\pi\ran,      & \text{ if } M[a, b] < \infty.  \end{cases} 
\end{equation}
Let $y_{0} \in \I$ be an arbitrary point.     Then 
\begin{equation}
\label{e1:g/xi-limit:b}
\lim_{y\to b} \frac{g(y) }{ \xi(y)} = \lim_{(w,y)\to (b,b)} \frac{g(y) - g(w) }{\xi(y) - \xi(w)}= \lim_{y\to b} \frac{g(y) - g(y_{0}) }{\xi(y) - \xi(y_{0})} = \bar c(b).
\end{equation}  
\end{lem} 

At first glance, the limiting value in \lemref{lem2-Bg/Bxi:b} when $(w,y) \to (b,b)$
 is a bit surprising in the case that $b < \infty$ and $M[a, b]< \infty$.  The points $w$ and $y$ become very close to each other and one would intuitively expect the cycle lengths to be quite short.  However, $b$ being natural with $M[a, b]< \infty$ means that $b$ is non-attracting.  This, in turn, implies that, starting from $w$, a point close to $b$,   a substantial fraction of the paths of  the process $X_{0}$ must wander below $w$ before reaching the level $y$, yielding {\em in the limit}\/  \eqref{e1:g/xi-limit:b} the average of $c$ with respect to the stationary distribution $\pi$ on the entire interval $\I$.

\section{Impulse Control Problem}\label{sect:classical impulse}

In this section, we consider the long-term average impulse control problem of maximizing $J(R; p, K, \gamma)$ with fixed parameters $p$, $K$ and $\gamma$.  For notational simplicity, we omit the superscript $R$ in $X^{R}$ throughout the section and suppress the parametric dependence of the functions. 

\subsection{Existence of an Optimal Policy}\label{subsect-existence}
Inspired by our investigation of optimal inventory control problems in \cite{helm:18,HelmSZ-25}, we consider the $(w, y)$-policy $R^{(w,y)}$ of \defref{thresholds-policy},  which resets the state level to $w$ whenever it reaches or initially exceeds level $y$.  We have observed in Section \ref{sect-preliminaries} that the long-term average reward of the policy $R^{(w,y)}$ is given by $F(w, y)$ of \eqref{e:F_K}.  However, under the   policy $\mathfrak R$, $X = X_{0}$. 
Moreover, since $b$ is natural, we can regard $\mathfrak R$ as the limit of $(w, y)$ policy when $y\to b$, where $w\in \I$ is an arbitrarily fixed point.  {  Now in view of Table 7.1 on p.~250 of \cite{KarlinT81}, we have 
\begin{equation}  \label{eq:xi_limit_b}
\lim_{x\to b} (\xi(x) - \xi(y_{0}) ) = \int_{y_{0}}^{b} M[a, v] dS(v) \ge  \int_{y_{0}}^{b} M[y_{0}, v] dS(v) =\infty   \quad \forall y_{0}\in \I. 
\end{equation} 
Therefore, under Conditions  \ref{diff-cnd} and \ref{c-cond}, we can use Lemmas \ref{lem-limit:Bid/Bxi-a} and  \ref{lem2-Bg/Bxi:b}, along with \eqref{eq:xi_limit_b} to derive   
\begin{align}\label{e:reward-zero-policy}
\nonumber J(\mathfrak R) & = \liminf_{t\to \infty}t^{-1} \EE\left[ \int_0^t c(X_0(s))\, ds\right] = \lim_{y\to b} F(w, y)\\ & = \lim_{y\to b}  \left(\frac{p(y- w)}{\xi(y)- \xi(w)}- \frac{K }{\xi(y)- \xi(w)} + \frac{\gamma ( g(y) - g(w) )}{\xi(y)- \xi(w)} \right) = \gamma \bar c(b). 
\end{align} 
}

As a result of these observations, the requirement imposed on the parameters $p$, $K$ and $\gamma$  that it is better to intervene than not can be expressed in the following functional form.

\begin{cnd}\label{cond-interior-max}
There exists a pair $(\wdt w, \wdt y) \in \cR$ so that $F(\wdt w, \wdt y) > \gamma \bar c(b)$, where $\bar c(b)$ is defined in \eqref{eq:c-bar(b)}
\end{cnd}

   Condition \ref{cond-interior-max}  indicates that there exists at least one $(w, y)$-policy that outperforms the do-nothing policy $\mathfrak R$.   

We now turn to the maximization of $F$ over pairs $(w,y)\in \cR$, a nicely structured fractional optimization problem.

\begin{prop}\label{prop-Fmax} Assume Conditions \ref{diff-cnd}, \ref{c-cond},  
and \ref{cond-interior-max} hold.
 Then  there exists a   pair $(w^*,y^*)\in \cR$ so that 
 \begin{equation}
\label{e-Fmax}
F(w^{*}, y^{*}) = \sup_{(w, y)\in \cR} F(w, y). 
\end{equation} 
In other words, the function $F$ has an optimizing pair $(w^*,y^*)\in \cR$. Furthermore, an optimizing pair satisfies the following (rearranged) first-order conditions: 
\begin{itemize}
  \item[(i)] If $a$ is a natural point, then every optimizing pair  $(w^*,y^*) \in \cR$  satisfies $a < w^{*} < y^{*} < b$ and
\begin{equation}  \label{e-1st-order-condition}  
F(w^{*}, y^{*}) = \sup_{(w, y)\in \cR} F(w, y)=  h(w^{*}) = h(y^{*}),
\end{equation} 
where the function $h: \I\mapsto \R$ is defined by 
\begin{equation}
\label{e-h-fn-defn}
h(x; p, \gamma) = h(x) := \frac{\gamma g'(x) +p}{\xi'(x)}, 
 \quad x\in \I.
\end{equation}
  \item[(ii)] If $a$ is an entrance point, then an optimizing pair  $(w^*,y^*) \in \cR$ may have $w^*=a$; in such a case, we have 
 \begin{equation}\label{e2-1st-order-condition}  
h(a) \ge   {F(a, y^{*})} = \sup_{(w, y)\in \cR} F(w, y) = h(y^{*}). 
\end{equation} 
But if $w^*>a$, \eqref{e-1st-order-condition} still holds.  
\end{itemize} 

  \end{prop} 
{  The proof of Proposition \ref{prop-Fmax} follows a similar approach to that in \cite{helm:18}, analyzing the boundary behavior of  $F$, and is given in Appendix \ref{sect:Appen-B}.}

In preparation for the analyses that follow, it is necessary to establish certain properties of $h$  and its  relationship with $F$.
Detailed  calculations using \eqref{e:xi-derivatives}, \eqref{e:g-derivatives}, and \eqref{e:1/s-identity} reveal that for any $x\in \I$, recalling   the revenue rate  function $r$ defined in \eqref{e:r_p}, we have
 \begin{align} \label{sect 2-e-h-defn} 
 h(x)  = \frac{s(x) (\int_{a}^{x} \gamma c(u)\, dM(u) + \frac p{s(x)})}{s(x) M[a,x]} = \frac{\int_{a}^{x} (\gamma c(u) + p \mu(u) )\, dM(u) }{M[a,x]}= \frac{\int_{a}^{x}   r(u)\,  m(u)\, du}{M[a, x]},
\end{align}
 and 
\begin{align}\label{e:h'-expression}  
h'(x) &=  \frac{m(x)}{M[a,x]} (r(x) - h(x))  = \frac{m(x)}{M^{2}[a,x]} \int_{a}^{x} [r(x) - r(u)]\, dM(u).
\end{align}  

The representation of $h$ in \eqref{sect 2-e-h-defn} will be shown in \sectref{sect:singular} to be related to the payoff of a particular singular control policy.

Next, for any 
$(w, y) \in \cR$, since  $K > 0$ and $\xi'(x)> 0$ for all $x\in \I$, we apply the generalized  mean value theorem to observe that 
\begin{align} \label{e:F<ell}
 F(w, y) 
 & <   \frac{\gamma ( g(y)-g(w) )+ p(y-w)}{\xi(y)-\xi(w)}  = \frac{\gamma g'(\theta) + p }{\xi'(\theta)}   = h(\theta) \le \sup_{x\in \I} h(x), 
\end{align} 
where $\theta \in (w, y)$. 

The next lemma presents the limiting behavior of  the function $h$ at $b$.  

 \begin{lem}\label{lem-ell-limit at b} 
 Assume Condition \ref{c-cond}. Then $\lim_{x\to b}h(x ) = \gamma \bar c(b)$.  
 \end{lem}  

 \begin{proof} 
 Note that  $\lim_{x\to b} \frac{p}{\xi'(x)} = 0$ thanks to  \eqref{e-sM-infty}. Therefore it remains to show that   $\frac{\gamma g'(x)}{\xi'(x)}$ converges as $x\to b$ to  $\gamma c(b)$ (if $M[a, b] =\infty$) or $\lan \gamma c, \pi\ran$ (if   $M[a, b] < \infty$). To this end,  we note that   the expressions for $\xi'$ and $g'$ in \eqref{e:xi-derivatives} and \eqref{e:g-derivatives}  allow  us to write 
\begin{equation}  \label{g'-xi'}
\frac{\gamma g'(x)}{\xi'(x)}= \frac{\int_{a}^{x} \gamma c(u)\, dM(u)}{ M[a, x]};    
\end{equation}  
as $x\to b$,  the last expression converges to  $\gamma c(b)$ (if $M[a, b] =\infty$) or $\lan \gamma c, \pi\ran$ (if   $M[a, b] < \infty$) thanks to \eqref{e1.1-lem35-pf} and \eqref{e1.2-lem35-pf} in the proof of Lemma  \ref{lem2-Bg/Bxi:b}. 
\end{proof}

 Motivated by Proposition \ref{prop-Fmax}, we now consider the impulse reward potential  
\begin{equation} \label{e:G_p}
G(x): = F^{*} \xi(x) - \gamma g(x),\ \ x\in \I,
\end{equation}  
where $ F^{*}  = F(w^{*}, y^{*})$ is the maximum value of the  function $ F$ defined  in \eqref{e:F_K}, and $(w^{*}, y^{*})$ is an optimizing   pair  for  $ F$.    The potential $G$ can be used to enlarge the class of admissible policies (see \remref{rem1-transversality}) and is intimately related to the relative value function for the impulse control problem.

 First, observe that  $G$ is a strictly increasing function. To see this, we write 
$$G'(x) = \xi'(x) \left(F^{*} - \frac{\gamma g'(x)}{\xi'(x)}\right) = \xi'(x) (F^{*} - h_{0}(x)), $$ 
where $h_{0}$ is defined in \eqref{e-h-fn-defn} with $p=0$. Using \eqref{e:h'-expression} and Condition \ref{c-cond}, $h_{0}$ is an increasing  function. \lemref{lem-ell-limit at b} further says that $\lim_{x\to b} h_{0}(x) = \gamma \bar c(b)$.  This, together with Proposition~\ref{prop-Fmax}, implies that  $F^{*} - h_{0}(x) > 0$ for any $x\in \I$.  Recall that $\xi'(x) > 0$ and hence $G'(x)  > 0$ for any $x\in \I$. This establishes the claim that  $G$ is  strictly increasing.  

 Next, using \eqref{e:AxiAg} and the definition of $F^{*} $, we can immediately verify that $G$ is twice continuously differentiable and satisfies 
\begin{equation}   
\label{e:AG_BG}
\begin{cases}
   A G(x) + \gamma c(x) - F^{*} = 0,  &  \ \forall  x \in \I,   \\ 
  G(w) +  p(y-w) -   K - G(y) \le  0 ,    &  \ \forall (w,y) \in \cR, \\
  G(w^*) +  p(y^*-w^*) -   K - G(y^*) =  0. & 
\end{cases}  
\end{equation}  
Note that, defining $ \mathcal M v(x) : = \sup_{w \le x}\{ v(w)+ p(x-w) -K\}$,  $(G, F^{*})$ is a solution to the  quasi-variational inequalities (QVI)
\begin{equation}\label{e:qvi} 
\max\{ A v(x) + \gamma c(x) -\lambda, \mathcal M v(x) - v(x) \} =0, \quad \forall x\in \I
\end{equation} 
for the long-term average impulse control problem \eqref{e:reward-fn}. 

We also point out that the third line of \eqref{e:AG_BG} can be written as
$$p B\id(w^*,y^*) + \gamma  Bg(w^*,y^*) - F^* B\xi(w^*,y^*) = K,$$
which is a balance equation exhibiting a zero-profit equilibrium in which the expected excess value exactly compensates for the fixed cost.

\comment{\begin{defn}\label{defn-Ap} \footnote{If we keep \defref{admissible-for-p}, then this definition can be deleted, though the references need to be updates.}
Let ${\A}_p \subset {\cal A}$ 
consist of   impulse control policies $Q = (\tau,Y)\in \A$ for which 
the transversality condition 
\begin{equation}
\label{eq-transversality}
\liminf_{t\to \infty} \liminf_{n\to \infty}t^{-1} \EE[G_{p}(X(t\wedge \beta_{n}))] \ge  0
\end{equation}
holds, where for each $n\in \NN$,  $\beta_{n}: = \inf\{t\ge 0: X(t) \not\in (a_{n}, b_{n} )\}$, where $\{a_{n}\}\subset (a, b)$ is a decreasing sequence with $\lim_{n\to\infty} a_{n} =a$ and $\{b_{n}\}\subset (a, b)$ is an increasing sequence with $\lim_{n\to\infty} b_{n} =b$. Note that $\lim_{n\to\infty} \beta_{n} =\infty$ a.s.  because  $a$ is non-attracting and $b$ is natural thanks to Condition \ref{diff-cnd}. 
\end{defn}}

\begin{lem}\label{lem-Gp-limit} Assume Conditions \ref{diff-cnd}, \ref{c-cond}, and \ref{cond-interior-max} hold. Then for any $R\in  \AI$, we have  
\begin{equation}  \label{eq-transversality}
  \liminf_{t\to \infty} \liminf_{n\to \infty}t^{-1} \EE[G(X(t\wedge \beta_{n}))] \ge  0,
  \end{equation}
 where $\beta_n$ is given in \defref{admissible-policy}(iv)(b).
\end{lem} 

\begin{proof} We establish this transversality condition in two cases. 
  
  First, if $a$ is an entrance point, then $\xi(a)> -\infty$ thanks to Table~6.2 on p.~232 of \cite{KarlinT81}. This, together with the fact that $\xi$ is monotone increasing, implies that   $ F^{*} \xi(x)$ is uniformly bounded from below. Since $g(x) \le 0$ for $x\le x_{0}$ and $\gamma > 0$, $G(x) = F^{*} \xi(x) - \gamma g(x)  $  is uniformly bounded from below on $(a, x_{0}]$. We next examine the behavior of $G$ at the right boundary $b$.  Using Lemma \ref{lem2-Bg/Bxi:b} and Proposition \ref{prop-Fmax}, we have 
  \begin{displaymath}
    \lim_{x\to b} \left(F^{*} - \frac{\gamma g(x)}{\xi(x)}\right)  > 0.
    \end{displaymath} 
    Thus for all $x$ in a neighborhood of $b$, $G(x) =(F^{*} - \frac{\gamma g(x)}{\xi(x)} ) \xi(x) > 0$. Therefore it follows  that $G$ is uniformly bounded from below on $(a, b)$. For each $n\in \NN$, let $\beta_{n}$ be as in \defref{admissible-policy}. 
    Since $G$ is uniformly bounded from below on $(a, b)$, and recalling  that $\beta_{n}\to \infty$ a.s. as $n\to \infty$,  we can apply Fatou's lemma to see that 
    \begin{displaymath}
   \liminf_{t\to \infty} \liminf_{n\to \infty}t^{-1} \EE[G(X(t\wedge \beta_{n}))] \ge  \liminf_{t\to \infty}  t^{-1} \EE[G(X(t))] \ge 0,
   \end{displaymath}
   establishing the desired  transversality condition.  This leads to \eqref{eq-transversality}.

 Second, if $a$ is a natural point, we can write 
 \begin{align}  \label{eq-Gp-bound}
    \nonumber G(x) & = F^{*} \xi(x) - \gamma g(x) =  [F^{*} \xi^+(x) - \gamma g^+(x)]  + \gamma g^-(x) -  F^{*} \xi^-(x)\\ 
    & =   [F^{*} \xi(x) - \gamma g(x)] I_{\{x > x_0\}}  + \gamma g^-(x) -  F^{*} \xi^-(x).
  \end{align}
   Thanks to Lemma \ref{lem2-Bg/Bxi:b} and Proposition \ref{prop-Fmax}, there exists a $y_0 \in (x_0, b)$ so that $F^{*} -\frac{\gamma g(x)}{\xi(x)} > \frac{F^*- \gamma \bar c(b)}{2} > 0$ for all $x\in (y_0, b)$. Note that $\xi(x) >0$ for $x > x_0$ and $F^{*} \xi(x) - \gamma g(x)$ is uniformly bounded for $x\in [x_0, y_0]$. Therefore it follows that there exists a constant $C> 0$ so that 
 \begin{align*}
     [F^{*} \xi(x) - \gamma g(x)] I_{\{x \ge  x_0\}}& = [F^{*} \xi(x) - \gamma g(x)] I_{\{x \in[  x_0, y_0]\}} + \left[F^{*} -\frac{\gamma g(x)}{\xi(x)}\right] \xi(x) I_{\{x > y_0\}} \ge -C.
   \end{align*} 
Plugging this observation into \eqref{eq-Gp-bound} and noting that $g^-(x) \ge 0$, we  derive
\begin{displaymath} 
    G(x) \ge -C + \gamma g^-(x) -  F^{*} \xi^-(x) \ge -C   -  F^{*} \xi^-(x), \quad \forall x\in (a, b).
\end{displaymath} 
  For any $R\in \AI$,   \eqref{eq-xi-transversality} holds, which, together with the last displayed equation, implies \eqref{eq-transversality}. The proof is complete. 
\end{proof}

\begin{prop}\label{prop-optimal-control}
Assume Conditions \ref{diff-cnd}, \ref{c-cond},  
and \ref{cond-interior-max} hold. Then for any admissible policy
 $R\in \A_{\rm Imp}$, 
  we have
\begin{equation}  \label{e-reward-bdd}   
J(R) \le  F^{*} = F(w^{*}, y^{*}), 
\end{equation} 
and the $(w^{*}, y^{*})$-strategy is an optimal policy, where $(w^{*}, y^{*})\in \cR$ is a maximizing pair for the function $F$ defined in \eqref{e:F_K}, whose existence is derived in Proposition \ref{prop-Fmax}. 
\end{prop}

\begin{proof} We have observed that for any $R\in  \AF$, $J(R) = J(\mathfrak R) = \gamma \bar c(b)$, which is less than   $F^{*}$ thanks to Proposition \ref{prop-Fmax}.  The rest of the proof is focused on impulse controls with an infinite number of interventions.  
 
Fix an arbitrary admissible policy $R\in   \AI$, and denote by $X$ the controlled process. For each $n\in \NN$, define $\beta_{n}$  as in  Definition \ref{admissible-policy}.  By the It\^o formula and \eqref{e:AG_BG}, we have 
\begin{align*} 
\EE&_{x_0}[G(X(t\wedge \beta_{n}))] & \\
& = G(x_0)   + \EE_{x_0} \left[\int_{0}^{t\wedge \beta_{n}} AG(X(s))\, ds + \sum_{k=1}^{\infty} I_{\{\tau_{k}\le t\wedge \beta_{n}\}} BG(X(\tau_{k}-), X(\tau_{k}))\right] 
 \\
& =  G(x_0) + F^{*}\EE_{x_0} [t\wedge \beta_{n}]  
- \EE_{x_0} \left[\int_{0}^{t\wedge \beta_{n}} \gamma c(X(s))ds - \sum_{k=1}^{\infty} I_{\{\tau_{k}\le t\wedge \beta_{n}\}}  BG(X(\tau_{k}-), X(\tau_{k}))\right].
\end{align*} 
Rearrange the terms 
to obtain 
 \begin{align*} 
&  G(x_0)  +  F^{*}\EE_{x_0} [t\wedge \beta_{n}] -    \EE_{x_0}[G(X(t\wedge \beta_{n}))] \\ 
&\ \quad  =   \EE_{x_0} \left[\int_{0}^{t\wedge \beta_{n}} \gamma c(X(s))ds + \sum_{k=1}^{\infty} I_{\{\tau_{k}\le t\wedge \beta_{n}\}} [G(X(\tau_{k}-)) - G( X(\tau_{k}))] \right] .
\end{align*} 
Since $\gamma c\ge 0$ and $G(X(\tau_{k}-)) - G( X(\tau_{k})) \ge 0$ for each $k\in \mathbb N$ because $G$ is a strictly increasing function, 
 passing to the limit as $n\to \infty$, we have from the monotone convergence theorem that 
 \begin{align} \label{eq25} 
   \nonumber G(x_0)   & +  F^{*}t  - \liminf_{n\to\infty}   \EE_{x_0}[G(X(t\wedge \beta_{n}))]\\ 
   & =   \EE_{x_0} \left[\int_{0}^{t } \gamma c(X(s))ds + \sum_{k=1}^{\infty} I_{\{\tau_{k}\le t \}} [G(X(\tau_{k}-)) - G( X(\tau_{k}))] \right] \\ 
  \nonumber & \ge \EE_{x_0} \left[\int_{0}^{t } \gamma c(X(s))ds + \sum_{k=1}^{\infty} I_{\{\tau_{k}\le t \}} [p(X(\tau_{k}-) -   X(\tau_{k})) -K] \right],  
\end{align}  where the last inequality follows from the second line of \eqref{e:AG_BG}. 
Now divide both sides by $t$ and then let $t\to \infty$ to obtain 
\begin{align*}  
F^{*} &- \liminf_{t\to\infty} \liminf_{n\to\infty} \frac1t \EE_{x_0}[G(X(t\wedge \beta_{n}))] \\  
& \ge  \limsup_{t\to\infty} \frac 1t \EE_{x_0} \left[\int_{0}^{t } \gamma c(X(s))ds + \sum_{k=1}^{\infty} I_{\{\tau_{k}\le t \}}  [p(X(\tau_{k}-) -   X(\tau_{k})) -K] \right].
\end{align*}  
Finally we   use  \eqref{eq-transversality}    to obtain 
\begin{equation}\label{e:reward-bdd-1}  
 F^{*} \ge \limsup_{t\to\infty} \frac 1t \EE_{x_0} \left[\int_{0}^{t } \gamma c(X(s))ds + \sum_{k=1}^{\infty} I_{\{\tau_{k}\le t \}} [p(X(\tau_{k}-) -   X(\tau_{k})) -K] \right].
\end{equation} 
 This establishes \eqref{e-reward-bdd}.  On the other hand, we have shown in Section \ref{sect-preliminaries} that $R^{(w^*,y^*)}\in \A_{\rm Imp}$ with $J(R^{(w^*,y^*)}) = F^*$.   Hence  $R^{(w^*,y^*)}$ is an optimal policy. This completes the proof.  
\end{proof}

\begin{rem}\label{rem1-transversality}  
When $a$ is a natural boundary, $\xi(a)= -\infty$ and thus it follows from \eqref{e:g/xi-limit:a}, Condition \ref{cond-interior-max}, and Proposition \ref{prop-Fmax} that \begin{displaymath}
  \lim_{x\to a} G(x) = \lim_{x\to a}\left (F^{*} - \frac{\gamma g(x)}{\xi(x)} \right) \xi(x) =  \lim_{x\to a} (F^{*} -  \gamma c(a)) \xi(x) =-\infty. 
  \end{displaymath} 
   The  transversality condition \eqref{eq-xi-transversality} is therefore introduced to verify the optimality in $\A_{\rm Imp}$ of the $(w^{*}, y^{*})$-impulse strategy.  In addition, a close examination of the proof of Proposition \ref{prop-optimal-control} reveals that  \eqref{eq-xi-transversality} can be relaxed. Indeed, equation \eqref{e:reward-bdd-1} holds as long as \eqref{eq-transversality} is true. Moreover,  by a similar argument as that for the proof of Proposition 4.2 of \cite{helm:18}, we can show that for any $(w, y)\in \cR$, the $(w,y)$-policy defined in Definition \ref{thresholds-policy} satisfies \eqref{eq-transversality}. 
For policies having infinitely-many interventions, introduce the less restrictive transversality condition as an alternative to \defref{admissible-policy}(iv)(b):\smallskip
\begin{description}
\item[\defref{admissible-policy}(iv)(c)]  if $a$ is a natural boundary, $\tau_k < \infty$ for all $k\in \mathbb N$ and the inequality \eqref{eq-transversality} holds, in which the sequence $\{\beta_n\}$ is given in \defref{admissible-policy}.
\end{description}\smallskip
Define the larger class of admissible policies $\A_G$ to be those policies satisfying \defref{admissible-policy}(i), (ii), (iii), and (iv)(a) or (iv)(c).  
Then, the $(w^{*}, y^{*})$-strategy is optimal in the class $\A_G$.  
Notice that since $G$ is defined using $F^*$, it is necessary to solve the optimization of $F$ over $(w,y) \in \R$ before one can determine the class $\A_G$.  However, defining admissibility with \defref{admissible-policy}(iv)(b) uses the exit time potential $\xi$ in the transversality condition \eqref{eq-xi-transversality}, which is determined from the fundamentals of the underlying diffusion model.  Thus, in principle, one can determine admissibility of a policy $R \in \A_{\rm Imp}$ before solving any optimization problem.
\end{rem}

\comment{
\begin{rem}\label{rem-A_p-policies} {\blue \footnote{We need to revisit this remark also.}
We make several remarks about the class $\A_{p}$ here. First, Lemma \ref{lem2-transversality} below indicates that under Conditions \ref{diff-cnd}, \ref{c-cond},  
and \ref{cond-interior-max},  for any $R\in \A_{p}$, the resulting controlled process $X$ necessarily satisfies \eqref{e2-transversality} and \eqref{e:kappa<z0}.   Second, when $a$ is an entrance boundary, $\A_{p} =\A$ thanks to Remark \ref{rem1-transversality}. Finally, we note that $\A_{p}\neq \emptyset$ for any $p > 0$ because it contains the ``$(w, y)$-policy'' of \defref{thresholds-policy}. This can be established  by a similar argument as that for the proof of Proposition 4.2 of \cite{helm:18}.
}
\end{rem}}

\subsection{Uniqueness}\label{subsect-uniqueness} 
Proposition \ref{prop-optimal-control} establishes the existence of an optimal strategy in  $\A_{\rm Imp}$  for the long-term average reward \eqref{e:reward-fn}  under Conditions \ref{diff-cnd}, \ref{c-cond}, and \ref{cond-interior-max}; moreover, the optimal strategy is of $(w,y)$-type with the levels given by the maximizing pair $(w^{*}, y^{*})$ for the function $F$.  This section shows that such a maximizing pair  is unique under the additional Condition \ref{3.9-suff-cnd}.  

{  
The next lemma establishes the unimodal behavior of the function $h$ as a consequence of \cndref{3.9-suff-cnd}.  To help visualize the result, the functions $r$ satisfying \cndref{3.9-suff-cnd} and $h$ are represented graphically in Figure~1; the symbols $\wdh x$, $\wdt x$ are as in \cndref{3.9-suff-cnd} and $\wdh y$ is defined below.
}
\begin{figure} [!ht]

\begin{subfigure}[b]{0.45\textwidth}
       \centering
		\begin{tikzpicture}
		 \begin{axis}[ legend pos= north east, axis lines = center, domain=0:6.1, xlabel=$x$,  xmin=-1.0, xmax=6, ylabel=$r(x)$,ymin=-.5,ymax=3.0, ticks=none, samples=50,]
	 \draw [blue, xshift=0cm,thick] plot [smooth] coordinates { (0,0.23) (0.3077086327, 0.7) (0.9689640690, 1.78) (1.670,2.2)  }; 
	  \draw[blue, xshift=0cm,thick] plot [smooth] coordinates {(2.35006, 2.203)  (3.2, 1.99) (4.72, 0.99) (5.99,0.081)};
	 \draw[blue,thick, xshift=0cm]  (1.6659,2.203)  --  (2.35, 2.2054);

		\addplot [blue, mark = |] coordinates {(1.7,0)}; \draw (1.7,0) node [below] {\small {$\wdh x$}};
		\addplot [blue, mark = |] coordinates {(2.35,0)}; \draw (2.35,0) node [below] {\small {$\wdt x$}};
		 \draw (0,0) node [anchor=north east] {\small {$a$}};
		  \draw (5.9,0) node [below] {\small {$b$}};
		\end{axis}		
		\end{tikzpicture} \caption{}
	\end{subfigure} 
\ \ 
  \begin{subfigure}[b]{0.45\textwidth}
       \centering
		\begin{tikzpicture}
		 \begin{axis}[ legend pos= north east, axis lines = center, domain=0:6.1, xlabel=$x$,  xmin=-1.0, xmax=6, ylabel=$h(x)$,ymin=-.5,ymax=3.0, ticks=none, samples=50,]
	 \draw [blue, xshift=0cm,thick] plot [smooth] coordinates { (0,0.12) (0.3077, 0.574) (0.96, 1.52)  (2.4,2.509)(3.1,2.43) (3.5, 2.18) (3.7,2.0)(4.82, 0.97)(5.23,.65) (5.99,0.13)};
	
		
		
		\addplot [blue, mark = |] coordinates {(2.53,0)};
     \draw (2.53,0) node [below] {\small {$\wdh y$}};
		 \draw (0,0) node [anchor=north east] {\small {$a$}};
		  \draw (5.9,0) node [below] {\small {$b$}};
		\end{axis}		
		\end{tikzpicture}\caption{}
	\end{subfigure}  
	\caption{(A) The function $r$.\qquad  \qquad \qquad  \qquad (B)  The function $h$.}\vspace{-0.1in}
		\label{fig2}
	\end{figure}

{  
Using the integral representation \eqref{e:h'-expression} of $h'$, $r$ being strictly increasing up to $\wdh x$ implies that $h$ is increasing on this interval as well.  Moreover by \propref{prop-Fmax}, there exists an optimal pair $(w^*,y^*)$ such that $h(w^*) \geq F^* = h(y^*)$ and hence there exists some $w^* < \wdt y < y^*$ for which $h'(\wdt y ) = 0$. Consequently,  we can now  define  
\begin{equation}  \label{e-y-hat-p-defn}
\wdh y_{p,\gamma} = \wdh y : = \min \{x \in \I: h'(x) = 0\}.
\end{equation} 
Note that $\wdh y > \wdt x$, where $\wdt x$ is the value specified  in \cndref{3.9-suff-cnd},  and that $\wdh y$ is the smallest extreme point of $h$ in $\I$. 
}

\begin{lem}\label{lem-h-new}
 Assume Conditions \ref{diff-cnd},    \ref{c-cond}, \ref{3.9-suff-cnd}, and \ref{cond-interior-max}  hold. Then $h$ is strictly increasing on $(a, \wdh y )$ and strictly decreasing on $(\wdh y, b)$. 
\ Consequently the optimizing pair $(w^{*}, y^{*}) \in \cR$ of Proposition \ref{prop-Fmax} is unique. 
\end{lem}

\begin{proof}  We observed above that $h$ is strictly increasing on $(a, \wdh y)$. It remains to show that $h$ is strictly decreasing on $(\wdh y, b)$. To this end, we use \eqref{e:h'-expression} to compute for $x > \wdh y$: 
\begin{align*} 
h'(x)   & = \frac{m(x)}{M[a,x]^{2}} \int_{a}^{x} [r(x) - r(u)]\, dM(u) \\
& = \frac{m(x)}{M[a,x]^{2}} \left[ \int_{a}^{\wdh y} [r(\wdh y) - r(u) +r(x) - r(\wdh y) ]\, dM(u)  + \int_{\wdh y}^{x} [r(x) - r(u)]\, dM(u)\right] \\
& <  \frac{m(x)}{M[a,x]^{2}}  \int_{a}^{\wdh y} [r(\wdh y) - r(u)]\, dM(u)  
  = h'(\wdh y) =0,
\end{align*} 
where the inequality above follows from the facts that  $x > \wdh y >\wdt x $ and that $r$ is strictly   decreasing on $(\wdt x, b)$. The proof is complete. 
\end{proof}

The following proposition follows directly from Propositions \ref{prop-Fmax}  and \ref{prop-optimal-control} and Lemma  \ref{lem-h-new}.  

\begin{prop}\label{prop-sec3-uni+existence} 
Assume Conditions \ref{diff-cnd},  \ref{c-cond}, \ref{3.9-suff-cnd},  and \ref{cond-interior-max} hold.    
Then there exists a unique maximizing pair $(w^{*}, y^{*}) \in \cR$  for the function $F$ of \eqref{e:F_K}. In addition,   the $(w^{*}, y^{*})$-policy  is an optimal admissible policy. 
\end{prop}

\comment{ \begin{proof}
Under Conditions \ref{diff-cnd} and \ref{cond-interior-max}, Proposition \ref{prop-Fmax} says that there exists a maximizing pair $(w^{*}_{p}, y^{*}_{p}) \in \cR$  for the function $F_{p}$ for which \eqref{e-1st-order-condition} holds. This together with \eqref{e:h'-expression} and the assumption that $r_{p} $  is  strictly increasing  on $(a, \wdh x_{p})$ implies that there exists an $ x > \wdh x_{p}$ so that $h_{p}'(x) =0$. Furthermore, we can define $y_{p} : = \min\{x \in \I: h_{p}'(x) =0\}$ as in \eqref{e-y-hat-p-defn}; note that $y_{p}> \wdh x_{p}$. Since  $r_{p} $ is decreasing on  $(\wdh x_{p}, b)$ by assumption, it is decreasing on $(y_{p}, b)$. In other words, Condition \ref{new-cond-Kurt} holds. Therefore, it follows from Lemma \ref{lem-h-new} that the maximizing pair $(w^{*}_{p}, y^{*}_{p})  $  for  $F_{p}$ is unique and   the assertions concerning $h_{p}$ of Lemma \ref{lem-h-new} hold. The optimality of  the $(w^{*}_{p}, y^{*}_{p})$-threshold strategy  follows from Proposition \ref{prop-optimal-control}. 
\end{proof}} 

\subsection{Probabilistic Cell Problem and Overtaking Optimality}\label{sect-stoch-cell}

 We briefly digress from the long-term average problem \eqref{e:reward-fn} to consider other approaches to infinite-horizon impulse control. For notational simplicity, for any $R\in \A_{\rm Imp}$, we define   the total expected reward up to time $T$
\begin{align*}
  J_T(x; R) : = \EE_x \left[\int_0^T \gamma c(X(s))\, ds+ \sum_{k=1}^{\infty} I_{\{\tau_k\le T\}} (p(X(\tau_k-) - X(\tau_k)) -K)  \right], 
\end{align*} where $x\in \I$ and $T> 0$. Often one wishes to maximizes the performance functional over a large time horizon; i.e., to maximize $\liminf_{T\to\infty} J_T(x;R)$ over all admissible policies $R$. However, as the running reward function $c$ is nonnegative, such a problem may be ill-posed in the sense that $\liminf_{T\to\infty} J_T(x;R) = \infty$ for many $R\in \A_{\rm Imp}$.  To address this issue, we consider a modified performance criterion in which a linear function of time is subtracted from $J_T(x;R)$. More specifically,
we consider the  infinite-horizon optimization problem of finding an income rate $\lambda$ and a finite-valued function $V$ which satisfies 
\begin{equation}  \label{e-value-cell-pr} 
V(x) := \sup_{R\in \A_{\mathbb T}}\liminf_{T\to\infty} [J_T(x;R) -\lambda T],  \qquad \forall x \in \I, 
\end{equation} 
where $\A_{\mathbb T}$ is the class of all $(w,y)$-policies introduced in Definition \ref{thresholds-policy}; a mnemonic for $\mathbb T$ is that these are thresholds policies. 
   With $\lambda T$ serving as a benchmark,  we can think of \eqref{e-value-cell-pr} as a  maximization problem for the long-term accumulated relative reward under impulse control.  Further, we seek an optimal policy $R \in \A_{\mathbb T}$ which achieves the maximal values for every $x \in \I$.

Such a problem is motivated by similar problems arising in Statistics and the so-called probabilistic cell problem studied in the context of the LQG problem; see, for example, \cite{BayrJ:25,JianJSY:24}. Similar to the LQG setting, in which the probabilistic cell problem is related to the ergodic Bellman equation, we demonstrate that for the impulse control model studied in this paper, the solution to \eqref{e-value-cell-pr} is closely connected to the function $G$ defined in \eqref{e:G_p}.  In fact, recall that $(G, F^*)$ solves the QVI \eqref{e:qvi} and that $F^*$  is the optimal value for   the long-term average   problem \eqref{e:reward-fn}. 
Proposition \ref{prop:overtaking} shows that if $a$ is an entrance boundary, then the value function $V$ of \eqref{e-value-cell-pr} is given by the impulse reward potential $G$, adjusted by its mean under the invariant measure resulting from the optimal $(w^*,y^*)$-policy; further adjustment is needed when the initial state $x$ is above the threshold $y^*$.

The rate $\lambda$ in \eqref{e-value-cell-pr} cannot be chosen arbitrarily.  To see this, assume the conditions of \propref{prop-sec3-uni+existence} hold, let $F^*=F(w^*,y^*)$ be the resulting optimal value and $R^* = R^{(w^*,y^*)}$ be the optimal policy.  If $\lambda < F^*$, then   
\begin{align*}
    \liminf_{T\to\infty}  [J_T (x, R^*) - \lambda T]
    = \infty, \quad \forall x\in\I.
  \end{align*} 
  On the other hand, if $\lambda > F^*$, then for any admissible policy $R\in \A_{\rm Imp}$, we have 
  \begin{align*}
  \liminf_{T\to\infty} 
  [J_T (x, R) - \lambda T]= -\infty. 
  \end{align*} 
Therefore, the only reasonable choice is $\lambda = F^*$.

Assume the conditions of \propref{prop-sec3-uni+existence}, let $(w^*,y^*)$ be the unique optimizer of the function $F$ of \eqref{e:F_K} and, as before, denote $F(w^*,y^*) = F^*$.  We now define piecewisely the function $\wdt G$ as follows:
\begin{align}\label{e:u_def}
    \wdt G(x) = \begin{cases}  
   F^* \xi(x) - \gamma  g(x), & x < y^*, \\  
p(x -w^*) - K +   G(w^*), & x \ge y^*.
\end{cases}
  \end{align}  
Notice that $\wdt G(x) = G(x)$ for $x \le  y^*$, where  the function $G$ is defined in \eqref{e:G_p}, but then $\wdt G$ grows linearly for $x \geq\ y^*$. Hence we can regard $\wdt G$ as a modified impulse reward potential. 

{    The next proposition imposes the restriction that the left boundary $a$ be an entrance boundary so that the function $G$ is bounded below.  This condition is used when takng limits in the paragraph before \eqref{e2:cell-problem}.  The proof does not hold in general when $G$ is unbounded below in a neighborhood of $a$.
}

 \begin{prop} \label{prop:overtaking}
  Assume  Conditions \ref{diff-cnd}, \ref{c-cond}, \ref{3.9-suff-cnd}, and  \ref{cond-interior-max}  hold and that $a$ is an entrance boundary.  Let $\wdt G$ be defined by \eqref{e:u_def} with $(w^*,y^*)$ being the unique optimizer of $F$.  Let $R^*$ denote the optimal $(w^*,y^*)$-policy and $\nu^{R^*}$ be the invariant measure induced by $R^*$.  Then the solution to problem \eqref{e-value-cell-pr} is the triplet  
$$ (F^*, \wdt G(\cdot) - \lan \wdt G, \nu^{R^*}\ran, R^*).$$ 
Therefore  the value function $V=\wdt G - \lan \wdt G, \nu^{R^*}\ran$ is the deviation of the modified impulse reward potential from its mean under $\nu^{R^*}$.
\end{prop}  

\begin{proof} We have observed 
in  Proposition \ref{prop-optimal-control} that the   policy  $R^*$ is optimal for   \eqref{e:reward-fn} with optimal payoff $F^*  $. We now show that the value function $V(\cdot)$ of \eqref{e-value-cell-pr} is given by $\wdt G(\cdot) - \lan \wdt G, \nu^{R^*}\ran$ and that $R^*$ is also optimal for \eqref{e-value-cell-pr}.  In view of  \eqref{e:F_K} and Proposition \ref{prop-sec3-uni+existence}, for any $R\in \A_{\mathbb T}$ with $R^{(w, y)}\neq R^*$, we have  $J(R^{(w, y)}) = F(w, y) < F^*$.
 Therefore it follows that 
 \begin{align}\label{e1:cell-problem}
    \liminf_{T\to\infty} 
    \big[J_T(x, R^{(w, y)})- F^* T\big] = -\infty. 
  \end{align} 
On the other hand, for the $R^*$-policy, denoting by $X$ the controlled process, the same calculations as those in deriving in \eqref{eq25} give  
  \begin{align}\label{e0:jump-G}   \nonumber   
  G(x)  &  +  F^{*}T  - \liminf_{n\to\infty}   \EE_{x}[ G(X(T\wedge \beta_{n}))]  \\ &
    =   \EE_{x} \left[\int_{0}^{T} \gamma c(X(s))ds + \sum_{k=1}^{\infty} I_{\{\tau_{k}\le T \}} [ G(X(\tau_{k}-)) -  G( X(\tau_{k}))] \right].
  \end{align}  
  We proceed in two cases. 
  
 \begin{itemize} \item[(i)] If $x < y^*$, then under the policy $R^*$, $X(\tau_{k}-) =y^*$ and $X(\tau_{k}) = w^*$ for all $k\in \mathbb N$. In addition, the last equaltion of \eqref{e:AG_BG} implies that 
\begin{align}\label{e1:jump-G}\nonumber 
   G(X(\tau_{k}-)) -  G( X(\tau_{k})) & =  G(y^*) -  G(w^*) = p(y^* - w^*) - K\\ &= p(X(\tau_k-) - X(\tau_k)) -K, \quad  k \in \NN.
  \end{align} 
Plugging \eqref{e1:jump-G} into \eqref{e0:jump-G} yields
\begin{align*} 
 G(x) & + F^* T -\lim_{n\to\infty} \EE_x[ G(X(T\wedge \beta_n))] \\ & =  \EE_{x} \left[\int_{0}^{T } \gamma c(X(s))ds + \sum_{k=1}^{\infty} I_{\{\tau_{k}\le T \}} [p(X(\tau_k-) - X(\tau_k)) -K] \right].\end{align*} 
  Rearrange the above displayed equation to obtain
  \begin{align}\label{e:J_T=G}
    J_T(x; R^*)- F^* T  & =    G(x) -\lim_{n\to\infty} \EE_x[  G(X(T\wedge \beta_n))] = \wdt G(x) -\lim_{n\to\infty} \EE_x[   G(X(T\wedge \beta_n))],  
\end{align} where the last equality holds because $\wdt G (x) = G(x)$ for all $x\le  y^*$. 
  
  \item[(ii)] If $x \ge y^*$, then $X$ jumps to $w^*$ at time $\tau_1 = 0$ and subsequently behaves as if it started from $w^*$. Therefore, 
\begin{align}\label{e2:jump-G} 
 \nonumber G(X&(\tau_{1}-)) -   G( X(\tau_{1}))  =   G(x) -   G(w^*) \\ \nonumber & = p(x - w^*) - K +  G(x) -  G(w^*)-[p(x - w^*) - K] \\    & 
= p(X(\tau_1-) - X(\tau_1)) -K +   G(x) -   G(w^*)-[p(x - w^*) - K];    
\end{align}
\eqref{e1:jump-G} still holds for all $k \ge 2$.  Using \eqref{e2:jump-G} and \eqref{e1:jump-G} in \eqref{e0:jump-G}, we derive 
\begin{align*}
    G&(x)     +  F^{*}T  - \liminf_{n\to\infty}   \EE_{x}[ G(X(T\wedge \beta_{n}))] \\ &  =  \EE_{x} \left[\int_{0}^{T } \gamma c(X(s))ds + \sum_{k=1}^{\infty} I_{\{\tau_{k}\le T \}} [p(X(\tau_k-) - X(\tau_k)) -K] \right] \\ & \quad +  G(x) -   G(w^*)-[p(x - w^*) - K]. 
\end{align*} Using the definition of $\wdt G$ given in \eqref{e:u_def}, a rearrangement of the above displayed equation  leads to 
  \begin{align*}
    J_T(x; R^*)- F^* T  & =  G(w^*) + [p(x - w^*) - K]   -\lim_{n\to\infty} \EE_x[ G(X(T\wedge \beta_n))]\\ &  = \wdt G(x) -\lim_{n\to\infty} \EE_x[   G(X(T\wedge \beta_n))]; 
  \end{align*} again establishing \eqref{e:J_T=G}
\end{itemize}

  Since $a$ is an entrance boundary, as observed in the proof of Lemma \ref{lem-Gp-limit}, the function $G$ is uniformly bounded from below.   On the other hand,  since  the process $X$ is bounded above by $x_0\vee y^*$, the random variables $  G(X(T\wedge \beta_n)), n \in \mathbb N $ are uniformly bounded. Hence  $\lim_{n\to\infty} \EE_x[ G(X(T\wedge \beta_n))] = \EE_x[ G(X(T))]$. Furthermore, since $X$ possesses the invariant measure $\nu^{R^*}$, we have $\lim_{T\to\infty} \EE_x[ G(X(T))] = \lan  G, \nu^{R^*}\ran$.  Since $\nu^{R^*}$ has support on $[a, y^*]$ and $G$ and $\wdt G$ coincide on this     interval,  we have $\lan  G, \nu^{R^*}\ran = \lan  \wdt G, \nu^{R^*}\ran$. Therefore taking the limit as $T \to \infty$ in \eqref{e:J_T=G}  yields 
  \begin{align}\label{e2:cell-problem} 
  \liminf_{T\to\infty}  [J_T(x; R^*)- F^* T ] = \wdt G(x) - \lan   G, \nu^{R^*}\ran= \wdt G(x) - \lan \wdt G, \nu^{R^*}\ran.
  \end{align} 
  A combination of \eqref{e1:cell-problem} and \eqref{e2:cell-problem} yields $V(\cdot)= \wdt G(\cdot) - \lan \wdt G, \nu^{R^*}\ran$ and that $R^*$ is an optimal policy for \eqref{e-value-cell-pr}. This completes the proof.
\end{proof} 

\begin{rem}
  As a side remark, we note that the function $\wdt G$ defined in \eqref{e:u_def} satisfies $\wdt G \in C^1(\I) \cap C^2(\I\backslash \{y^*\})$.  Moreover, $\wdt G$ and $F^*$ satisfy the system \eqref{e:AG_BG} and the QVI \eqref{e:qvi}.   In fact, in view of the  smooth pasting method for ergodic singular and impulse control problems, c.f. \cite{HelmSZ-17,JackZ-06,KunwXYZ-22,LiangLZ-25}, we may regard $(\wdt G, F^*)$ as the {\em usual} solution to \eqref{e:qvi}.
\end{rem}

\begin{rem}
  Now consider problem \eqref{e-value-cell-pr} over a  larger class   of   impulse control policies, namely, $\A_{\mathrm S}$, the class of all admissible policies $R$ for which the controlled process $X^R$ has a stationary distribution, then  it is not clear whether the value function $V$ and an optimal policy $R$ remain the same as those over $\A_{\mathbb T}$.  

  Suppose $R\in \A_{\mathrm S}$ is arbitrary and denote by $\nu^R$ the invariant measure induced by $R$. If the long-term average reward of $R$ is less than $F^*$, then using similar calculations as those preceding \propref{prop:overtaking}, we have 
   $ \liminf_{T\to\infty} 
     [J_T(x;R) - F^* T]= -\infty.
  $ 
  On the other hand, if the long-term average reward of $R$ equals $F^*$, then for any $n\in \mathbb N$  and $T> 0$, with $\beta_n$ similarly defined  as in Definition \ref{admissible-policy}, using the same calculations as those in deriving \eqref{e2:cell-problem}, we have   
  \begin{align*}
    \limsup_{T\to\infty}[J_T(x; R) - F^* T]
    &  \le \wdt G(x) - \liminf_{T\to\infty} \lim_{n\to\infty} \EE_x[\wdt G(X(T\wedge \beta_n))] \\ 
    &  \le \wdt G(x) - \liminf_{T\to\infty}  \EE_x[\wdt G(X(T))] \\
    & \le   \wdt G(x) - \lan \wdt G, \nu^R\ran,
  \end{align*} 
where the second inequality follows from Fatou's lemma and the assumption that $a$ is an entrance boundary,  which ensures that $\wdt G$ is bounded from below, and which together with the continuity of $\wdt G$ and the Portmanteau theorem imply the last inequality.  Therefore it follows that 
\begin{align}\label{e:cell-upper-bd}
  V(x) \le \wdt G(x) - \inf_{R\in \A_{\mathrm S}: J(R) = F^*} \lan \wdt G, \nu^R\ran.
\end{align}  
Since it is not clear whether the above inequality is in fact an equality,  we leave this issue as an interesting topic for future research.
\end{rem}

  We now turn to another perspective on impulse control over an infinite horizon: overtaking optimality. This important concept in infinite-horizon optimization was first introduced in \cite{Ramsey-28} and has applications in economics,   operation research, and control theory. 
  Weaker formulations were subsequently developed in \cite{BrockH:76, Weiz-65, Gale-67} and many other works; see also \cite{DockJLS-00, CarlsonHL:91} for extensive surveys and comprehensive lists of references on overtaking optimality. Following the terminology in \cite{HuangYZ-21}, we say that an admissible policy $\wdh R\in \A_{\rm Imp}$ is overtaking optimal starting from $x\in \I$ if for each $R \in \A_{\rm Imp}$,  we have 
\begin{align*}
  \limsup_{T\to\infty} [J_T(x; R) - J_T(x, \wdh R)] \le 0.
\end{align*} 
\begin{cor}\label{cor-overtaking}
  Assume  Conditions \ref{diff-cnd}, \ref{c-cond}, \ref{3.9-suff-cnd}, and  \ref{cond-interior-max}  hold. Then the $R^* = R^{(w^*,y^*)}$-policy is overtaking optimal in $\A_{\mathbb T}$ starting from each $x\in \I$.
\end{cor} \begin{proof}
  This follows directly from \eqref{e:F_K} and Proposition \ref{prop-sec3-uni+existence}. Indeed, for any $(w, y)\neq (w^*, y^*)$, denoting the corresponding policies by $R = R^{(w, y)}$ and $R^* = R^{(w^*, y^*)}$, respectively, we have
  \begin{align*}
    \lim_{T\to\infty} \frac1T J_T(x; R) = F(w, y) < F^* = \lim_{T\to\infty} \frac1T J_T(x; R^*), \quad\forall x\in \I.
  \end{align*} This implies that for any $0 < \e  < \frac{F^* - F(w, y)}{2} $, there exists some $T_0 > 0$ so that for all $T \ge T_0$, we have \begin{align*}
    J_T(x; R) < (F(w, y) + \e) T < (F^* - \e) T < J_T(x; R^*).  \tag*{\qedhere} 
  \end{align*}
\end{proof}

\section{The Singular Control Problem} \label{sect:singular}

This section is devoted to   the singular control problem \eqref{e:singular-reward}. 

\begin{defn}[Singular Admissibility] \label{singular-admissible-policy}
A process $Z = \{Z(t), t\ge 0\}$ is an {\em admissible singular control}\/ if it is nondecreasing, right-continuous, and $\{\F_t\}$-adapted process with $Z(0-) =0$, the controlled state process $X^Z$ satisfies \eqref{e:SDE-XZ} with $X^Z(t) \in \E$ for all $t \geq 0$, and when $a$ is a natural boundary,
\begin{align}\label{e:sing-transversality} 
\lim_{t\to\infty}\limsup_{n\to\infty}t^{-1}\EE [\xi^-(X^Z(t\wedge \beta_n)) ] =0,
 \end{align} 
where $\beta_n: =\inf\{t\ge0: X^Z(t) \notin (a_n,b_n)\}$, with $\{a_n\}$ and  $\{b_n\}$ as in Definition \ref{admissible-policy}(iv)(b).  Denote the set of admissible singular controls by $\A_{\rm Sing}$.

\end{defn}

We introduce the important subclass of reflection policies.

\begin{defn}[$x$-Reflection Policies] \label{local-time-policy}
For each $x\in \I$,  let $L_x$ denote the local time process at $x$ of the uncontrolled process $X_0$.  Define the {\em $x$-reflection policy} $Z_x = \{Z_x(t): t \geq 0\}$ by
$$Z_x(t)= (x_0-x)^+ + L_x(t), \qquad t \geq 0.$$ 
The policy $Z_x$ keeps the controlled process $X^{Z_x}$ within $[a, x]$ by reflection at $x$; if $x_0 > x$,  an initial jump to $x$ occurs at time $t=0$.   
\end{defn}


\begin{lem} \label{Zx-admissible}
For each $x \in \I$, the singular policy $Z_x$ is admissible in the sense of \defref{singular-admissible-policy}. Moreover, the following assertions are true: 
\begin{itemize}
  \item[(a)] the long-term average reward for the reflection policy $Z_x$ is 
\begin{equation} \label{e:Jp(Z_x)=hp}  
\wdh  J(Z_x) = 
\int_a^b r(u) \pi_x(u)\, du, 
\end{equation}
in which $r$ is given by \eqref{e:r_p}, and \begin{align}\label{e:singular-inv-meas} \pi_x(u) =  
  \frac{m(u)}{M[a,x]}I_{[a, x]}(u);  
\end{align} 
  \item[(b)]  $\lim_{x \to a} \wdh J(Z_x) = r(a) = \gamma c(a) + p \mu(a)$ and 
   $\lim_{x \to b} \wdh J(Z_x) = \gamma \bar c(b)$. 
 \end{itemize}
\end{lem}
\begin{rem} \label{impulse-singular-connection}
The value of $\wdh  J(Z_x)$ of \eqref{e:Jp(Z_x)=hp} 
gives a remarkable connection between the impulse and singular control problems. 
Using  the integral representation of the function $h$ given  in \eqref{sect 2-e-h-defn}, we note that $\wdh  J(Z_x) = h(x)$ for each $x\in \I$.  Thus the function $h$ can be interpreted as the long-term average reward  of the $x$-reflection policy for the singular control problem.  In particular, 
  the impulse optimality condition \eqref{e-1st-order-condition} indicates that the optimal value $F^* =F(w^*,y^*)$ for the impulse control problem in \eqref{e:reward-fn} is the value of the $y^*$-reflection policy for the singular problem and also for the $w^*$-reflection policy when $w^* > a$.
\end{rem}
\begin{proof}
It is easy to see that $Z_x$  is a nondecreasing, right-continuous, and $\{\F_t\}$-adapted process
  with $Z(0-) =0$.  Referring to Section~5 of Chapter~15 in \cite{KarlinT81}, the diffusion $X^{Z_x}$ possesses the invariant measure having density $\pi_x$ on $\E$.
Then, using the ergodicity for linear diffusions in Chapter II, Section 6 of \cite{boro:02}  and a similar argument as that in the proof of Lemma \ref{lem-(wy)admissible}, we can show that the controlled process $X^{Z_x}$ satisfies \eqref{e:sing-transversality} and hence $Z_x\in \AS$. 

 Using the invariant density $\pi_x$ of \eqref{e:singular-inv-meas} and (a) of Chapter II.6.37 in  \cite{boro:02}, we have 
\begin{align}\label{e:lta-c}
   \lim_{t\to\infty} t^{-1} \EE_{x_0}\left[ \int_0^t \gamma c(X^{Z_x}(s)) ds  \right] =\int_a^b \gamma c(u) \pi_x(u)\, du.
\end{align} 
Turning to the reward obtained by reflection, first observe that 
$$\lim_{t\to \infty} t^{-1} \EE_{x_0}[ p Z(t)] = \lim_{t\to \infty} t^{-1} \left(p (x_0-x)^+ + \EE_{x_0}[p L_x(t)]\right) = p\, \lim_{t\to \infty} t^{-1} \EE_{x_0}[L_x(t)].$$
To compute this limit, note that $X^{Z_x}(t)$ is bounded.  Thus we can apply Itô's formula to $\id(X^{Z_x}(t))$ and then take expectations to obtain
\begin{align*}
\EE_{x_0}[X^{Z_x}(t)] & = x_0 + \EE_{x_0}\left[\int_0^t \mu(X^{Z_x}(s)) ds  -(x-x_0)^+ -L_x(t)\right]. 
\end{align*} 
Now dividing both sides by $t$ and sending $t\to\infty$ and again using (a) of Chapter II.6.37 in \cite{boro:02}, we have 
\begin{align*}
  0& = \lim_{t\to\infty} t^{-1} \EE_{x_0}\left[\int_0^t \mu(X^{Z_x}(s)) ds  \right] - \lim_{t\to\infty} t^{-1} \EE_{x_0}[L_x(t)]  \\ &
   = \int_a^x \mu(u) \pi_x(u)\, du - \lim_{t\to\infty} t^{-1} \EE_{x_0}[L_x(t)],
\end{align*} 
Hence we have 
\begin{equation} \label{lta-exp-local-time}
\lim_{t\to\infty} t^{-1} \EE_{x_0}[L_x(t)] = \int_a^x \mu(u) \pi_x(u)\, du
\end{equation} 
and the assertion \eqref{e:Jp(Z_x)=hp} now follows by combining \eqref{e:lta-c} and \eqref{lta-exp-local-time}. 

The first limit in part   (b) follows immediately from the facts that $c$ and $\mu$ extend continuously at $a$ with finite values and the observation that the stationary distribution having density $\pi_x$ converges weakly to a unit point mass at $a$ as $x \to a$.  The same argument applies at $b$ once one notices that the stationary distributions having density $\pi_x$ converge weakly (as $x \to b$) to a unit point mass at $b$ when $M[a,b] = \infty$ and to the stationary measure $\pi$ of \propref{lem-mu-at-b}(iii) having density $\frac{m}{M[a,b]}$ when $M[a,b] < \infty$.  This argument also uses \lemref{lem-mu-at-b}(iii) and (iv) to eliminate any dependence on $\mu$ in the limit $\gamma \bar c(b)$.
\end{proof}

The equality of the long-term average expected local time with the mean drift rate under the stationary distribution in \eqref{lta-exp-local-time} seems surprising.  The local time increases only when $X^{Z_x}$ is at $x$, whereas the mean drift rate is determined by the values of $u < x$.  To have a stationary process, however, the increase in the mean drift must be counterbalanced by the mean decrease provided by the local time so \eqref{lta-exp-local-time} does intuitively make sense.

\comment{
The expression \eqref{e:Jp(Z_x)=hp} enables us to use weak convergence to establish the limiting values of $\wdh J(Z_x)$ as $x$ approaches the boundaries.

\begin{lem}  \label{Zx-asymptotics}
For each $x \in \I$, let $Z_x$ denote the reflection policy of \defref{local-time-policy} so that the process $X^{Z_x}$ remains in the interval $\E \cap [a,x]$ .  Then
\begin{itemize}

\end{itemize}
\end{lem}

\begin{proof}

\end{proof}} 
Recall that the $\wdh J(Z_x) = h(x)$.  We now impose a condition, similar in spirit to \cndref{cond-interior-max}, that ensures that there exists some active singular control policy (more precisely, the $Z_{\wdt x}$ policy) that outperforms the do-nothing policy $\mathfrak{R}$. 
  
\begin{cnd}  \label{cond-hp-int-max}
There exists some $\wdt x\in\I$ so that 
  \begin{equation}  \label{e:cond-hp-barc}
  h(\wdt x) > \gamma \bar c(b).
\end{equation}
\end{cnd} 
Note that if Conditions \ref{diff-cnd}, \ref{c-cond},  \ref{3.9-suff-cnd}, and \ref{cond-interior-max} hold, then Condition \ref{cond-hp-int-max} holds automatically since $h(\wdh y) > F^* > \gamma \bar c(b)$ by \eqref{e:F<ell} and \eqref{e-1st-order-condition} or \eqref{e2-1st-order-condition}.   

\begin{rem} 
{  The argument leading to \eqref{e-y-hat-p-defn}} and the proof of \lemref{lem-h-new} rely on \cndref{cond-interior-max} to establish that $h(w^*) \geq h(y^*) > \gamma \bar c(b)$ in order to have the existence of $\wdt y$ such that $h'(\wdt y) = 0$.   The existence of $\wdh y$ of \eqref{e-y-hat-p-defn} then follows.  \cndref{cond-hp-int-max} provides a similar strict inequality for the singular control problem.  Therefore {  the argument preceding \eqref{e-y-hat-p-defn}} and \lemref{lem-h-new} on the behavior of $r$ and $h$, respectively, remain valid when \cndref{cond-hp-int-max} is imposed instead of \cndref{cond-interior-max}.
\end{rem} 

\begin{prop}\label{prop-singular-optimal} Assume Conditions \ref{diff-cnd}, \ref{c-cond},  \ref{3.9-suff-cnd}, and \ref{cond-hp-int-max} hold.
   Then for any admissible singular control policy $Z$, we have $\wdh J(Z) \leq h(\wdh y)$. Moreover,  the $Z_{\wdh y}$ policy is an optimal admissible singular control policy. 
\end{prop}

\begin{proof} 
 As mentioned above, the validity of the argument leading to \eqref{e-y-hat-p-defn} gives the existence of $\wdh y$ as in \eqref{e-y-hat-p-defn}.  Define the function  
\begin{align}
  U(x) & = \begin{cases}  h(\wdh y) \xi(x)   -  \gamma g(x), & \text{ if }x \in (a, \wdh y], \\ 
  U(\wdh y) + p( x - \wdh y ), &  \text{ if } x  \in ( \wdh y, b). 
  \end{cases} \label{e:singular-u}
 \end{align} 
By examining the left and right hand limits of $U, U'$, and $U''$ at $\wdh y$, we can readily verify that    $U$ is  twice continuously differentiable on $\I$. We next show   that  the pair $(U, h(\wdh y))$ solves the associated QVI for the singular control problem 
  \begin{align}  \label{e:HJB-singular}
  \max\{A U(x) + \gamma c(x) - \lambda_0, -U'(x)+ p\} = 0, \quad x\in \I,
 \end{align} 
 where $h$ and $\wdh y$ are  defined respectively in \eqref{e-h-fn-defn} and \eqref{e-y-hat-p-defn}.
\comment{
Obviously $U$ is continuous.  Using the definition of $h$ in \eqref{e-h-fn-defn}, we have 
 $$U'(\wdh y-) =   h(\wdh y) \xi'(\wdh y)- \gamma g'(\wdh y)= p = U'(\wdh y+).$$ 
 Similarly, using \eqref{e:xi-derivatives}, \eqref{e:g-derivatives},  \eqref{e:h'-expression}, and the fact that $h'(\wdh y) = 0$, we have 
 \begin{align*}
  U''(\wdh y-) & = h(\wdh y) \xi''(\wdh y) - \gamma g''(\wdh y) \\ 
  & = h(\wdh y)\bigg(-\frac{2\mu(\wdh y)}{\sigma^2(\wdh y)} \xi'(\wdh y) + s(\wdh y)m(\wdh y)\bigg) - \gamma \bigg(-\frac{2\mu(\wdh y)}{\sigma^2(\wdh y)} g'(\wdh y) + s(\wdh y)m(\wdh y)c(\wdh y)\bigg)\\ 
  &  = -\frac{2\mu(\wdh y)}{\sigma^2(\wdh y)} [h(\wdh y)\xi'(\wdh y) - \gamma g'(\wdh y)] + \frac{2}{\sigma^2(\wdh y)} [h(\wdh y) - \gamma c(\wdh y)] \\ 
  & = \frac{2}{\sigma^2(\wdh y)} [- p \mu(\wdh y) + h(\wdh y) - \gamma c(\wdh y) ] \\ 
  & = \frac{2}{\sigma^2(\wdh y)} [h(\wdh y) - r(\wdh y) ] \\ 
  & = 0 =U''(\wdh y+).
 \end{align*} 
 Thus $U$ is twice continuously differentiable on $\I$. 

Next,  we show that $(U, h(\wdh y))$ is a solution to \eqref{e:HJB-singular}.} 
To this end, we note that   $A U(x) + \gamma c(x) -   h(\wdh y) =0$ for $x\in (a, \wdh y)$  and $U'(x) =p$ for $x\in(\wdh y, b)$.  For $x \in (a, \wdh y)$, using  the facts that $h(\wdh y) =\max\{h(x): x\in \I\}$, that $\xi'(x) > 0$ and \eqref{e-h-fn-defn}, we have $U'(x) = h(\wdh y)\xi'(x) - \gamma g'(x) \ge h(x)\xi'(x) - \gamma g'(x)  = p  .$ 
Thus $-U'(x) + p \le 0$ for $x\in (a, \wdh y)$. For $x \in (\wdh y, b)$, we compute
\begin{align*}
  A U(x) + \gamma c(x) -   h(\wdh y) = p \mu(x) + \gamma c(x) - h(\wdh y) \le  r(x)- h(x)  < 0,
\end{align*} 
where the last inequality follows from \eqref{e:h'-expression} and the fact that $h$ is strictly decreasing on   $ (\wdh y, b)$. This shows that $(U, h(\wdh y))$ is indeed a solution to \eqref{e:HJB-singular}.

 We now show that $\wdh J(X^Z) \leq h(\wdh y)$ for all admissible singular control processes $Z$.
 Fix an arbitrary admissible singular control $Z\in \A_{\mathrm {Sing}}$ and denote by $X$ the controlled process. For each $n\in \NN$, define $\beta_{n}$  as in \defref{singular-admissible-policy}.  By the It\^o formula, we have 
 \begin{align}\label{e0:verification-proof-lta}
		\nonumber	\EE_{x_0}[U(X(T\wedge \beta_{n})]) &= U(x_0) +\EE_{x_0}\left[ \int_{0}^{T\wedge \beta_{n}} A U(X(s))d s - \int_{0}^{T\wedge \beta_{n}} U'(X(s-))\, dZ({s})^{c}
       \right.  \\ & \qquad  \left.
  +\sum_{0\leq s \leq T\wedge \beta_{n}} \left[U(X(s)) - U(X(s-))\right]\right].
\end{align}
The HJB equation \eqref{e:HJB-singular} implies that $   U' (x) \ge p$. Note also that $\Delta X(s)  =  X(s) -   X(s-) = -\Delta Z(s)  \le 0$.
		Consequently, we can use  the mean value theorem to obtain 
\begin{equation}\label{e1:verification-proof-lta}
		U( X(s)) - U( X(s-)) = U'(\theta)\Delta X(s) \le -p \Delta Z(s),
\end{equation} 
where $\theta $ is between $X(s)$ and $X(s-)$. Note that \eqref{e:HJB-singular} also implies that $A U(x) \le -\gamma c (x) + h(\wdh y)$. Plugging this
 and \eqref{e1:verification-proof-lta} into \eqref{e0:verification-proof-lta} gives us
\begin{align}\label{e2:verification-proof-lta}
		\EE_{x_0}[U(X(T\wedge \beta_{n}))] &  \le  U(x_0) -\EE_{x_0}\left[ \int_{0}^{T\wedge \beta_{n}}  \gamma c(X(s))d s +  h(\wdh y) (T\wedge \beta_{n})  
     - pZ(T\wedge \beta_{n})\right].
\end{align} Using an argument analogous to that in the proof of Lemma \ref{lem-Gp-limit}, it follows from \eqref{e:sing-transversality} that  $$\liminf_{t\to\infty}\liminf_{n\to\infty} T^{-1} \EE_x[U(X(T\wedge \beta_{n}))] \ge 0.$$ Therefore, 
by rearranging terms of \eqref{e2:verification-proof-lta}, dividing both sides by $T$,  and then sending $T\to\infty$, we 
obtain from the monotone convergence theorem that
\begin{displaymath}
		\wdh J(Z) = \liminf_{T\to\infty}\frac1T \EE_{x_0}\left[ \int_{0}^{T} \gamma c(X(s))d s +  p Z(T)   \right] \le h(\wdh y).
\end{displaymath} 

 With $x = \wdh y$, we have from \eqref{e:Jp(Z_x)=hp} that $\wdh J(Z_{\wdh y}) = h(\wdh y)$ and hence the $Z_{\wdh y}$ policy  is an optimal singular control policy.
\end{proof}

\section{Sensitivity Results} \label{sect:sensitivity}
This section examines the sensitivity of optimal quantities to changes in the parameters $p$, $K$ and $\gamma$ for the impulse control problem.    We impose Conditions \ref{diff-cnd}, \ref{c-cond}, \ref{3.9-suff-cnd} and \ref{cond-interior-max} to guarantee the uniqueness of the optimal $(w^*,y^*)$-policy.  

This analysis follows the same general argument for each parameter and is more straightforward when the optimizer $w^* > a$.  The analysis when $a$ is an entrance boundary and $w^* = a$ requires more care but the argument follows the same line of reasoning.  We therefore conduct the sensitivity analysis for those models for which $w^* > a$.

The first result establishes the differentiability of the optimizers in each of the parameters.  Let $q \in \{ p, K, \gamma\}$ denote the parameter under consideration and, to emphasize the $q$-dependence, denote the function $F$ of \eqref{e:F_K} by $F(q):=F(q;w,y)$, the optimizing pair by $(w_q^*,y_q^*)$ and the function $h$ of \eqref{e-h-fn-defn} by $h(q) := h(x;q)$.  Note that $h$ does not involve the parameter $K$ so $h(x;K) = h(x)$ but when $q=p$ or $q=\gamma$, $h(x;q)$ varies as $q$ varies and the other parameters are held fixed.  

First observe that if Condition~\ref{cond-interior-max} holds for some $q_0$, the continuity of $F(q)$ implies that it also holds for all $q$ in some open interval $U_1$ containing $q_0$.  Also, Conditions~\ref{diff-cnd}, \ref{c-cond}, and \ref{3.9-suff-cnd} do not involve the parameter $q$.  Therefore by Propositions \ref{prop-Fmax} and \ref{prop-sec3-uni+existence}, the family of optimization problems having $q \in U_1$ has unique optimizers which satisfy the first-order conditions \eqref{e-1st-order-condition}.  Define the vector-valued function $\Psi: U_1 \to \cR$ by
$$\Psi(q) = (w_q^*,y_q^*) = \argmax_{(w,y)\in \cR} F_q(w,y).$$

\begin{lem}\label{lem1-prop41-pf}
Assume Conditions \ref{diff-cnd},   \ref{c-cond}, \ref{3.9-suff-cnd} hold. Let $q_0$ be such that Condition~\ref{cond-interior-max} holds and let $U_1$ be as above so that Condition~\ref{cond-interior-max} holds for all $q \in U_1$. Then $\Psi$ is continuously differentiable at $q_0$. 
\end{lem}

\begin{proof}
We consider the function $H: U_1 \times \cR \mapsto \R^{2}$ defined by 
\begin{displaymath}
H(q,w,y) = \begin{pmatrix} H_1(q,w,y) \\ H_2(q,w,y) \end{pmatrix} : = \begin{pmatrix} h(w;q) - F(q;w, y) \\ h(y;q) - F(q;w, y) \end{pmatrix}.
\end{displaymath} 
Note that $H$ is continuously differentiable with  $H(q_0, w^{*}_{q_0}, y^{*}_{q_0}) =0$ and, thanks to \eqref{e-1st-order-condition}, we have (for the Jacobian matrix in $(w,y)$) 
\begin{align*}
\mathscr J & (q_0,  w_{q_0}^{*}, y_{q_0}^{*}):  = \begin{pmatrix} \partial_{w} H_{1} & \partial_{y} H_{1} \\  \partial_{w} H_{2} & \partial_{y}F_{2} \end{pmatrix}(q_0, w^{*}_{q_0}, y^{*}_{q_0}) = \begin{pmatrix}
h'(w_{q_0}^{*};q_0) & 0 \\ 0 & h'(y_{q_0}^{*};q_0)
\end{pmatrix}.
\end{align*} 
In view of the monotonicity of $h(\cdot\,;q_0)$ derived in \lemref{lem-h-new}, we have $h'( w_{q_0}^{*};q_0 ) > 0$  and $h'( y_{q_0}^{*};q_0) < 0$.
Therefore ${\mathscr J}(q_0, w_{q_0}^{*}, y_{q_0}^{*})$ is an invertible matrix and hence we can apply the implicit function theorem to conclude that there exists an open neighborhood $U_2 \subset U_1$ containing  $q_0$ and a unique continuously differentiable function $\psi: U_2\mapsto \cR$ such that $\psi(q) = ( w_{q}^{*}, y_{q}^{*}) = \Psi(q)$ and $H_{1}(q,\psi(q)) =0, H_{2}(q,\psi(q)) =0$ for $q\in U_2$.  In particular, this gives the continuous differentiability of $\Psi$ at $q_0$ as desired. 
\end{proof}

Having established that the optimizers are differentiable in each parameter $q$, the Envelope Theorem implies that the value function $F^*(q) = F(q;w_q^*,y_q^*)$ and the supply rate expression $\mathfrak{z}^*_q  :=  (\frac{B \id}{B \xi})(w_q^*,y_q^*) $ are differentiable functions as well.  The next proposition summarizes the sensitivity to variation in the parameters of important quantities of the solution to the impulse control problem.  This summary is presented in table form wherein $\nearrow$ indicates that the quantity increases, while $\searrow$ represents a decrease, and {\em differentiation with respect to the parameter}\/ is denoted by a dot over the function.  Thus $h'(x;q) = \frac{dh}{dx}(x;q)$ while $\dot h(x;q) = \frac{dh}{dq}(x;q)$.  These sensitivity results vary one parameter while holding the remaining parameters fixed.

{  A mild condition on the drift $\mu$ is imposed to determine the sensitivity results with respect to the price parameter $p$.

\begin{cnd} \label{drift-cnd}
There exists $\wdh x_\mu \in \I$ such that $\mu$ is strictly increasing on the interval $(a, \wdh x_\mu)$ and is concave on $(\wdh x_\mu, b)$.
\end{cnd}
}

\begin{prop} \label{prop-sensitivity}
 Let $q \in \{ p, K, \gamma\}$.  Assume Conditions \ref{diff-cnd}, \ref{c-cond}, \ref{3.9-suff-cnd}, \ref{drift-cnd} hold and that $q$ is such that Condition~\ref{cond-interior-max} holds. Let the pair $(w_q^*,y_q^*)$ denote the maximizer of $F(q)$.  Then the following table exhibits the response of the quantities to an increase in the parameter.                                                                                                  
\begin{center}
\renewcommand{\arraystretch}{1.3}
\setlength{\tabcolsep}{10pt}

\begin{tabular}{|c||c|c|c|c|c|}
\hline
$q$ &  $F^*(q)$    &  $\dot F^*(q)$ &  $w_q^*$        & $y_q^*$      & $\mathfrak{z}^*_q$ \\  \hline \hline
$p$ & $\nearrow$ & $\nearrow$    & Indeterminate & $\searrow$ & $\nearrow$              \\ \hline
$K$ & $\searrow$ & $\nearrow$ & $\searrow$& $\nearrow$ & Indeterminate. \\ \hline
$\gamma$ & $\nearrow$  & $\nearrow$ & $\nearrow$  & $\nearrow$ & Indeterminate \\ \hline
\end{tabular}
\end{center}
\end{prop}
\medskip

\begin{rem}
We highlight some implications of these results.    For the price parameter $p$, the fact that $ \dot F^*(p)$ is monotone increasing implies that the value function $F^*(p)$ is convex. Furthermore, $\dot F^*(p) = \mathfrak{z}^*_p$ by the Envelope Theorem, so the nonnegativity of $\dot {\mathfrak {z}} ^*_p$  reflects a fundamental postulate of micro-economics, namely that the supply function of producers is increasing in the market price.  That is, the higher the market price the more producers are willing to expand their production. 
\end{rem}

It will be helpful to break the function $h$ of \eqref{e-h-fn-defn} into the linear combination of two functions:
$$h(x;q) = h(x) = \frac{\gamma g'(x) + p}{\xi'(x)} =: \gamma h_c(x) + p h_\mu(x), \qquad x \in \I.$$
This notation is derived from the facts (see \eqref{g'-xi'}, and \eqref{e:1/s-identity} with \eqref{e:xi-derivatives}, respectively) that for each $x \in \I$, 
$$h_c(x) = \frac{g'(x)}{\xi'(x)} = \frac{\int_a^x c(u)\, m(u)\, du}{ M[a,x] } \quad \text{and}\quad h_\mu(x) = \frac{1}{\xi'(x)} = \frac{\int_a^x \mu(u)\, m(u)\, dx }{ M[a,x]}.$$
Observe that $h_c$ is monotone increasing and that   $h'  :=  (p h_\mu  +  \gamma h_c)'$ evaluated at $w_q^*$ is positive, while $h'(y_q^*) < 0$.    Considering $h_\mu$, note that 
\begin{equation} \label{h-mu-prime}
h'_\mu(x) = \frac{m(x)}{M[a,x]^2}\int_a^x (\mu(x)- \mu(y))dM(y).
\end{equation} 
{  Hence $h_\mu$ is positive and strictly increasing on $(a, \wdh x_\mu)$ thanks to Condition \ref{drift-cnd}.} On the other hand,  \eqref{e-sM-infty} implies that $\lim_{x\to b} h_\mu(x) = 0$.  {  Therefore, there exists some $\wdt x_\mu$ at which $h'_\mu(\wdt x_\mu) = 0$.  Since $\mu$ is strictly increasing on $(a, \wdh x_\mu)$, \eqref{h-mu-prime} implies $h_\mu' > 0$ on this interval.  For $h'(\wdt x_\mu) = 0$ to occur, there must exist some $\bar x_\mu$ with $\wdh x_\mu \le  \bar x_\mu < \wdt x_\mu$ at which $\bar x_\mu$ is the largest maximizer of $\mu$ and $\mu(x) < \mu(\bar x_\mu)$ on $(\bar x_\mu, \wdt x_\mu)$.  The concavity of $\mu$ on $(\wdh x_\mu, b)$ then implies that $h_\mu$ is strictly decreasing on $(\wdt x_\mu, b)$ while \eqref{h-mu-prime} also implies $h_\mu$ is strictly increasing on $(a, \wdt x_\mu)$.
}


The proofs of the results in \propref{prop-sensitivity} are based on a common idea: (i) use the first-order conditions \eqref{e-1st-order-condition} of \propref{prop-Fmax} in conjunction with uniqueness from \propref{prop-sec3-uni+existence}; (ii) employ the
  Envelope Theorem and evaluate the derivatives with respect to $q$ of the expression which depends on $w^*_q$ and on $y^*_q$, and of the expression which depends on $F^*(q)$; and (iii) carefully analyze the resulting expressions and equations to identify the sign of $(h_\mu(w^*_q)  -  \mathfrak{z}^*_q)$.  We now present the proof of \propref{prop-sensitivity}, beginning with the simplest case to analyze.
  
\begin{proof}[Proof of \propref{prop-sensitivity}]
\noindent
{\bf i.} Consider the fixed cost parameter $K$, allowing this to vary but keeping $p$ and $\gamma$ fixed.  It follows from \lemref{lem1-prop41-pf} that
$$F^*(K) :=  F(K;w_K^*,y_K^*)  = \left(\frac{p  B \id}{B\xi}  +  \frac{\gamma  B g}{B\xi}    -   \frac{K}{B\xi}  \right)(w_K^*,y_K^*)$$
is differentiable in $K$. Hence, by the Envelope Theorem
\begin{equation} \label{Fdot-K}
\dot F^*(K)   =    -  \frac{1}{B\xi}(w_K^*,y_K^*)
\end{equation}
and since $B\xi(w_K^*,y_K^*) > 0$,  $F^*(K)$ is strictly monotone decreasing. Intuitively, the result is obvious! Moreover, it follows from the first-order condition \eqref{e-1st-order-condition} that
$$\dot F^*(K)  =  h'(w_K^*)  \dot w_K^*  =  h'(y_K^*)  \dot y_K^*.$$
Since $h'(w_K^*)$ is positive and $h'(y_K^*)$  is negative, the results pertaining to $w_K^*$ and $y_K^*$ follow.  For the final definitive result, as $w_K^*$ and $y_K^*$ move apart $\dot F^*(K)$ increases in value.  Finally, $B\id(w^*_K,y^*_K) = y^*_K - w^*_K$ increases and so does $B\xi(w^*_K,y^*_K)$ but the change in the ratio $\mathfrak{z}(w^*_K,y^*_K) = \frac{B\id}{B\xi}(w^*_K,y^*_K)$ cannot be determined in general.  This now verifies all of the results pertaining to an increase in the fixed price $K$, as presented in the table.
\medskip

\noindent
{\bf ii.} Now consider varying the price parameter $p$ while holding $K$ and $\gamma$ fixed.  We first prove that $y_p^*$ decreases as $p$ increases.  Implementing the first two steps, (i) and (ii), of the common idea (taking the derivative with respect to $p$ in the first-order condition and using the Envelope Theorem) yields 
\begin{equation} \label{envelope-analysis}
  \frac{dh(y^*_p;p)}{dp} = h'(y^*_p)  \dot y^*_p  +  h_\mu(y^*_p)  =  \dot F^*(p) = \mathfrak{z}_p(w^*_p,y^*_p)  =  \frac{B \id}{B \xi}(w^*_p,y^*_p) = \mathfrak{z}^*_p.
 \end{equation}
  Next, to be specific about idea (iii), observe the following facts: 
  \begin{itemize}
  \item[(a)] $h_c$ is a strictly increasing function on $\I$,
  \item[(b)] the first-order condition gives 
$$     (p h_\mu  +  \gamma h_c)(w^*_p)  =  (p h_\mu  +  \gamma h_c)(y^*_p),   \quad \text{with }w^*_p  <  y^*_p, \quad p > 0, \gamma \ge 0, $$
and 
\item[(c)] there exists a point $\theta_p$, with $w^*_p < \theta_p < y^*_p$, such that 
$(B \id / B \xi)(w^*_p,y^*_p)  =  h_\mu(\theta_p)$.
\end{itemize} 

Using (a) in the first-order condition (b) establishes that  $h_\mu(w^*_p)   >   h_\mu(y^*_p)$.   Thus since $h_\mu$ is strictly increasing on $(a,\wdt x_{\mu})$ and strictly decreasing on $(\wdt x_{\mu},b)$, the relation $w^*_p < \theta_p < y^*_p$ implies $h_\mu(\theta_p)  >   h_\mu(y^*_p)$.  Now using the second and sixth expressions in \eqref{envelope-analysis} yields
\begin{equation} \label{y-dot}
\dot y^*_p = \frac{ \mathfrak{z}^*_p - h_\mu(y^*_p)}{h'(y^*_p)} = \frac{h_\mu(\theta_p) - h_\mu(y^*_p)}{h'(y^*_p)}.
\end{equation}
Since the numerator is positive and $h_\mu'(y^*_p) < 0$, the optimal jump-from position $y^*_p$ is decreasing in $p$.

We now examine $\dot F^*(p)$.  Following a similar argument as for the first equation in \eqref{y-dot}, it follows that
\begin{equation} \label{w-dot}
\dot w^*_p = \frac{ \mathfrak{z}^*_p - h_\mu(w^*_p)}{h'(w^*_p)}.
\end{equation}
Using the equality of the third and fifth expressions in \eqref{envelope-analysis}, taking derivatives with respect to $p$, noting that $h_\mu(x) = \frac{1}{\xi'(x)}$ in the fourth equality and using \eqref{w-dot} in the last equality yields
\begin{align*}
\ddot F^*(p) & = \frac{d}{dp}\left(\frac{B\id(w^*_p,y^*_p)}{B\xi(w^*_p,y^*_p)}\right) \\
& = \frac{B\xi(w^*_p,y^*_p) (\dot y^*_p - \dot w^*_p) - (y^*_p - w^*_p) (\xi'(y^*_p) \dot y^*_p - \xi'(w^*_p) \dot w^*_p)}{(B\xi(w^*_p,y^*_p))^2} \\
& = \frac{1}{B\xi(w^*_p,y^*_p)} \left[ (\dot y^*_p - \dot w^*_p) - \mathfrak{z}(w^*_p,y^*_p) (\xi'(y^*_p) \dot y^*_p - \xi'(w^*_p) \dot w^*_p)\right] \\
& = \frac{\xi'(y^*_p)}{B\xi(w^*_p,y^*_p)} \left[ \dot y^*_p \left( h_\mu(y^*_p) - \mathfrak{z}^*_p \right) \right] - \frac{\xi'(w^*_p)}{B\xi(w^*_p,y^*_p)} \left[ \dot w^*_p \left( h_\mu(w^*_p) - \mathfrak{z}^*_p \right) \right] \\
& = \frac{\xi'(y^*_p)}{B\xi(w^*_p,y^*_p)} \left[ \dot y^*_p \left( h_\mu(y^*_p) - \mathfrak{z}^*_p \right) \right] + \frac{\xi'(w^*_p)}{B\xi(w^*_p,y^*_p)} \left[ \frac{ \left( h_\mu(w^*_p) - \mathfrak{z}^*_p \right)^2}{h'(w^*_p)} \right].
\end{align*}
Evaluating the signs of these terms, both $\dot y^*_p$ and $h_\mu(y^*_p) - \mathfrak{z}^*_p$ are negative while both $\xi'(y^*_p)$ and $B\xi(w^*_p,y^*_p)$ are positive so the first summand is positive.  Similarly, all terms in the second summand are positive which establishes that $\ddot F^*(p) > 0$.  This immediately shows that $\dot F^*_p$ is increasing and by \eqref{envelope-analysis} that $\mathfrak{z}^*_p$ is increasing.  Moreover since $\mathfrak{z}^*_p > 0$, it again follows from \eqref{envelope-analysis} that $F^*_p$ is increasing, establishing all of the definitive results in the table pertaining to sensitivity with respect to $p$.  

We now show that the sensitivity of $w^*_p$ to an increase in $p$ cannot be determined in general.  In \eqref{w-dot}, $h'(w^*_p) > 0$ so the sign of $\dot w^*_p$ is the same as the sign of $\mathfrak{z}^*_p - h_\mu(w^*_p) = h_\mu(\theta_p) - h_\mu(w^*_p)$, where $\theta_p$ is given by observation (c).  It is not obvious whether the last expression is positive or negative.  We examine two cases.
\medskip

\noindent 
{(\bf a)} {\em Suppose $w^*_p \geq \wdt x_{\mu}$.}\/  Since $h_\mu$ is strictly decreasing on $(\wdt x_{\mu}, b)$ and $\theta_p > w^*_p$, it follows that $h_\mu(\theta_p) - h_\mu(w^*_p) < 0$ and hence that $\dot w^*_p < 0$.
\medskip

\noindent
{(\bf b)} {\em Suppose $w^*_p < \wdt x_{\mu}$.}\/  This analysis compares $h_\mu$ on $[w^*_p,b)$ to the function $y \mapsto \mathfrak{z}(w^*_p,y) = \frac{B\id}{B\xi}(w^*_p,y)$ for $y > w^*_p$.  Several subcases need to be examined.
First, applying the extended mean value theorem, for each $y > w^*_p$, there is some $\theta_y$ with $w^*_p < \theta_y < y$ for which $\mathfrak{z}(w^*_p,y) = \frac{1}{\xi'(\theta_y)} = h_\mu(\theta_y)$.  Extending $\mathfrak{z}(w^*_p,y)$ by continuity to $y=w^*_p$, we have $\mathfrak{z}(w^*_p,w^*_p) = h_\mu(w^*_p)$.  
In addition, on the interval $(w^*_p, \wdt x_{\mu})$, $h_\mu$ is increasing so since $w^*_p < \theta_y$, it follows that $h_\mu(w^*_p) < \mathfrak{z}(w^*_p,y)$.  {\em If $\theta_p \in (w^*_p,\wdt x_{\mu})$}, then 
$$\mathfrak{z}^*_p = \mathfrak{z}(w^*_p,y^*_p) = h_\mu(\theta_p) > h_\mu(w^*_p)$$ 
and the expression $\dot w^*_p$ of \eqref{w-dot} is positive.  

However, it may occur that $\theta_p \geq \wdt x_{\mu}$.  So more analysis is required.  On the interval $(\wdt x_{\mu}, b)$, the function $h_\mu$ is decreasing and converges to $0$ as $y \to b$.  There is therefore some conjugate value $\wdt w_p < b$ for which 
\begin{equation} \label{tilde-w_p}
h_\mu(\wdt w_p) = h_\mu(w^*_p).
\end{equation}

Consider the case in which $\wdt x_{\mu} < \theta_p < \wdt w_p$.  Then 
$$h_\mu(w^*_p) < h_\mu(\theta_p) = \mathfrak{z}^*_p$$ 
and the expression for $\dot w^*_p$ is again positive.

But if $\theta_p > \wdt w_p$, then 
$$\mathfrak{z}^*_p = h_\mu(\theta_p) < h_\mu(\wdt w_p) = h_\mu(w^*_p)$$
and the expression for $\dot w^*_p$ is negative.

As this analysis demonstrates, the existence of $\theta_p$ from an application of the mean value theorem does not provide sufficient information on its location and hence the effect on $w^*_p$ of an increase in $p$ can only be determined using this analysis on a case by case basis.

The sensitivity analysis for the scaling parameter $\gamma$ follows a similar approach and is left to the reader.
\end{proof}

\begin{rem}
For models in which there are no subsidies, that is when $\gamma = 0$, the effects on $w^*_p$ and $\mathfrak{z}^*_K$ to increases in the parameters $p$ and $K$, respectively, can be resolved.  For example, the function $h$ of \eqref{e-h-fn-defn} is the same as $h_\mu$.  As a result, the value $\wdt w_p$ that satisfies \eqref{tilde-w_p} is given by $\wdt w_p = y^*_p$ and since $\theta_p \in (w^*_p,y^*_p)$, it follows that $\mathfrak{z}^*_p = h_\mu(\theta_p) > h_\mu(w^*_p)$.  Therefore, $\dot w^*_p > 0$ and $w^*_p$ increases as $p$ increases.

Turning to the sensitivity of $\mathfrak{z}^*_K$ to an increase in $K$, we analyze ${   \dot{\mathfrak{z}}^*_K}$ as follows.
\begin{align*}
\dot{\mathfrak{z}}^*_K & = \frac{d}{dK}\left(\frac{B\id(w^*_K,y^*_K)}{B\xi(w^*_K,y^*_K)}\right) \\
& = \frac{B\xi(w^*_K,y^*_K) ({\dot y}^*_K - {\dot w}^*_K) - B\id(w^*_K,y^*_K) (\xi'(y^*_k) {\dot y}^*_K - \xi'(w^*_K) {\dot w}^*_K }{(B\xi(w^*_K,y^*_K))^2} \\
& = \frac{1}{B\xi(w^*_K,y^*_K)} \left[ \dot y^*_K  (1 - \mathfrak{z}(w^*_K,y^*_K) \xi'(y^*_K) - \dot w^*_K (1 - \mathfrak{z}(w^*_K,y^*_K) \xi'(w^*_K) \right] \\
& = \frac{1}{B\xi(w^*_K,y^*_K)} \left[ \dot y^*_K  \xi'(y^*_K) \left(h_\mu(y^*_K) - h_\mu(\theta_p)\right) - \dot w^*_K \xi'(w^*_K) \left(h_\mu(w^*_K) - h_\mu(\theta_p)\right)  \right] ;
\end{align*}
the last equality utilizes the fact that $h_\mu(x) = \frac{1}{\xi'(x)}$ for every $x \in \I$ and from (c) above that $\mathfrak{z}_K(w^*_K,y^*_K) = h_\mu(\theta_p)$ since $h = h_\mu$ when $\gamma=0$.  Observe that $B\xi(w^*_K,y^*_K) > 0$, $\dot y^*_K > 0$, and $\xi'(x) > 0$ for all $x \in \I$, as well as $\dot w^*_K < 0$.  Again, since $h_\mu$ is strictly increasing on $(w^*_K,\wdt x_\mu)$ and is strictly decreasing on $(\wdt x_\mu, y^*_K)$, the facts that $w^*_K < \theta_p < y^*_K$ and $h_\mu(w^*_K) = h_\mu(y^*_K)$ by the first order condition, imply $h_\mu(y^*_K) - h_\mu(\theta_p) = h_\mu(w^*_K) - h_\mu(\theta_p) < 0$ and hence that $\dot{\mathfrak{z}}^*_K < 0$.  Thus for the special model without subsidies, higher fixed costs will increase the production volume and extend the (average) length of a production cycle in such a manner that the supply rate is reduced.
\end{rem}


\section{Impulse and Singular Control Connections}\label{sect-imp_sing}

This section explores some of the connections between the impulse and singular control problems.  \remref{impulse-singular-connection} already observed that the first-order conditions \eqref{e-1st-order-condition} for the impulse control problem relates the optimal value $F^*$ to the sub-optimal singular payoffs for reflection policies at the optimizers $w^* > a$ and $y^*$ for the impulse problem.  We now prove that the solutions to the impulse problems, parametrized by the fixed cost $K$, converge to the solution of the singular control problem as $K \to 0$.  The parameters $p$ and $\gamma$ are held constant throughout.

Interpreting \propref{prop-sensitivity} when $K$ decreases, both the continuation region $(a, w^*_{K})$   and  the intervention region $(y^*_{K}, b)$ expand. The next proposition highlights the intrinsic connection between the impulse control problem \eqref{e:reward-fn} and the singular control problem \eqref{e:singular-reward},  and formalizes the intuition that the latter is  the limiting case of the former when the fixed cost tends to zero. 

 \begin{prop}  \label{prop-K-to-0} 
 Assume Conditions \ref{diff-cnd}, \ref{c-cond},  \ref{3.9-suff-cnd}, and \ref{cond-interior-max} hold. Then 
 \begin{itemize}
  \item[{\em (a)}] As $K\downarrow 0$, the optimal impulse control value $F^*_K$ for \eqref{e:reward-fn} converges to the optimal singular control value $h(\wdh y)$ for \eqref{e:singular-reward}, in which $\wdh y$ is defined in \eqref{e-y-hat-p-defn}. 
\item[{\em(b)}] Denote by $(w^*_K,y^*_K)$ the unique maximizing pair for the function  
\begin{equation} \label{F-K-defn}
F(K) :=  F(K;w,y)  = \left (\frac{p  B \id}{B\xi}  +  \frac{\gamma  B g}{B\xi}    -   \frac{K}{B\xi}  \right)(w,y), \qquad (w,y) \in \cR.
\end{equation}
As $K\downarrow 0$, the optimal $(w^*_K,y^*_K)$-impulse control policy for \eqref{e:reward-fn} converges to the optimal $Z_{\wdh y}$ singular control policy for \eqref{e:singular-reward},
 and the invariant measure induced by the $(w^*_K,y^*_K)$-policy converges weakly to that induced by the $Z_{\wdh y}$ policy.
 \end{itemize} 
 \end{prop} 
 
\begin{proof}
We first show that as $K\downarrow 0$, $F^*_K \uparrow h(\wdh y)$. Since $F_K^*$ and $h(\wdh y)$ are respectively the optimal impulse control value for \eqref{e:reward-fn} and the optimal singular control value for \eqref{e:singular-reward}, this convergence gives assertion (a) of the proposition.  For any $\e > 0$, since $h$ is continuous and achieves its maximum value at the unique maximizer $\wdh y$, we can find $w_\e < \wdh y < y_\e$ such that $h(w_\e) = h(y_\e) = h(\wdh y) -\frac\e2$. Note that $h(x) > h(\wdh y) -\frac\e2 $ for all $x\in (w_\e, y_\e)$. Let $K_\e: = \frac\e2 B\xi(w_\e, y_\e) $. Since $\xi$ is strictly increasing,  we have $K_\e > 0$. Then, using the calculations in \eqref{e:F<ell}, we can write 
\begin{align*} 
\frac{\gamma Bg(w_\e, y_\e) + pB\id(w_\e, y_\e)}{B\xi(w_\e, y_\e)} = \frac{\gamma g'(\theta)+ p}{\xi'(\theta)} = h(\theta) > h(\wdh y) -\frac\e2,  
\end{align*} 
where $\theta \in (w_\e, y_\e)$. This together with the definition of $F_{K_\e}$ in \eqref{F-K-defn} and \propref{prop-sensitivity} imply that for any $0 < K < K_\e$, we have
   \begin{align*}
   F_K^* > F_{K_\e}^* \ge F_{K_\e}(w_\e, y_\e) =  \frac{Bg(w_\e, y_\e) + pB\id(w_\e, y_\e)}{B\xi(w_\e, y_\e)}- \frac{K}{B\xi(w_\e, y_\e)} > h(\wdh y) -\e. 
   \end{align*} 
On the other hand, $F_K^* < h(\wdh y)$ for all $K>0$ in view of \eqref{e:F<ell}.     Thus  the limit  $\lim_{K\downarrow 0} F_K^* = h(\wdh y)$ is established. This limit also implies that $w^*_K > a$ for all $K>0$ sufficiently small.

Now letting $K \downarrow 0$ in \eqref{e-1st-order-condition}, we have 
$$\lim_{K\downarrow 0} h(w_K^*) = \lim_{K\downarrow 0} h(y_K^*)= \lim_{K\downarrow 0} F_K^* = h(\wdh y).$$ 
This, in turn, implies that   both $w^*_{K}$ and $y^*_{K}$ converge to $\wdh y$  as $K\downarrow 0$. Then  it follows from Propositions \ref{prop-sec3-uni+existence} and \ref{prop-singular-optimal} that    the optimal $(w_K^*,y_K^*)$-impulse control policy converges to the optimal $Z_{\wdh y}$-singular control policy as $K\downarrow 0$. 
  Furthermore,  the convergence of $w^*_{K}$ and $y^*_{K}$  to $\wdh y$ implies that $\lim_{K\downarrow 0}\frac{S[w^*_{K}, y^*_{K}]}{B\xi(w^*_{K}, y^*_{K})} = \frac{1}{M[a, y_p]}$. Combined  with   \eqref{e:nu_density-wy} and \eqref{e:singular-inv-meas}, this shows that the invariant measure induced by the $(w_K^*,y_K^*)$-policy converges weakly to that induced by the $Z_{\wdh y}$-reflection policy. 
  \end{proof}

\begin{rem}\label{rem-comparison-singular} 
We observed in \eqref{e:Jp(Z_x)=hp} that $h(x)$ is the long-term average reward of the $Z_x$ policy for any $x\in \I$.   In particular, $h(\wdh y)$ is the optimal long-term average reward  for the singular control problem  \eqref{e:singular-reward}. On the other hand,  Proposition \ref{prop-Fmax}  shows that for any fixed cost $K>0$, the optimal impulse control value for \eqref{e:reward-fn} satisfies $F^*_K = h(y^*_K) \le h(w^*_K)$, and the $(w^*_K, y^*_K)$-policy is an optimal impulse control policy for \eqref{e:reward-fn}, where from \propref{prop-sec3-uni+existence}, $(w^*_K, y^*_K)$ is the unique maximizing pair for the function $F(K)$. 
  
Therefore,the optimal impulse control value for \eqref{e:reward-fn} equals the  long-term average reward of the $Z_{y^*_K}$ policy (and the $Z_{w^*_K}$ policy if $w_K^* > a$) for the singular control problem  \eqref{e:singular-reward}. 
Both $Z_{w^*_K}$ and $Z_{y^*_K}$ policies are suboptimal for \eqref{e:singular-reward}. The difference $h(\wdh y) - h(y^*_K)$ between the optimal singular control value and the optimal  impulse control value  measures the reduction in value due to the  presence of the positive fixed cost $K$. 

 Exploring this a little further, using \eqref{e-1st-order-condition} and the definition of $h$ in \eqref{e-h-fn-defn}, an application of the extended mean value theorem yields 
\begin{align*}
  h( y_K^*) &  = F(w_K^*,y_K^*) = \frac{\gamma Bg(w_K^*,y_K^*) + p B\id(w_K^*,y_K^*) - K}{B\xi(w_K^*,y_K^*)} \\ & = \frac{\gamma g'(\theta_1) + p}{\xi'(\theta_1)} - \frac{K}{B\xi(w_K^*,y_K^*)}=h(\theta_1) - \frac{K}{B\xi(w_K^*,y_K^*)} 
\end{align*}
in which $w_K^* < \theta_1 < y_K^*$.  Hence $\frac{K}{B\xi(w_K^*,y_K^*)} = h(\theta_1) - h(y_K^*)$.  This implies that the inclusion of a positive fixed cost $K$ in the model creates a difference between the singular and impulse optimal values of 
$$\hat J(\wdh y) - F_K^* = h(\wdh y) - h(\theta_1) + h(\theta_1) - h(y_K^*) = ( h(\wdh y) - h(\theta_1) ) + \frac{K}{B\xi(w_K^*,y_K^*)}.$$
The difference $h(\wdh y) - h(\theta_1)$ represents an additional reduction in value beyond the rate at which the fixed cost would be continually assessed.  This reduction arises from using impulse policies, which form a subclass of the singular policies.  The difference $\hat J(\wdh y) - F_K^*$ can thus be interpreted to be a penalty that is incurred in models having the positive fixed cost $K$.  
This intrinsic and profound connection between impulse and singular controls has, to the best of our knowledge, not been explored in the existing literature.
\end{rem}

\begin{rem} \label{quantifying-fixed-cost}
Now consider the singular control problem from an operational viewpoint.  When there is no positive fixed cost, the optimal singular control policy is $Z_{\wdh y}$ with value $h(\wdh y)$.  Implementing such a  reflection policy, however, is impossible in many applications.  The upshot of \propref{prop-K-to-0} is that we can approximate the optimal reflection policy $Z_{\wdh y}$ by an $\epsilon$-optimal $(\wdt w,\wdt y)$-policy.  To see this, select $\wdt y > \wdh y$ with $h(\wdt y) > h(a)$ and such that $h(\wdh y) - h(\wdt y) < \eps$.  Let $\wdt w < \wdh y$ be the conjugate value defined by requiring $h(\wdt w) = h(\wdt y)$ and note that $\wdt w > a$.  Setting $\wdt F = h(\wdt y)$, we can solve for the constant $\wdt K$ such that
$$\wdt F = \frac{\gamma Bg(\wdt w,\wdt y) + p B\id(\wdt w,\wdt y) - \wdt K}{B\xi(\wdt w,\wdt y)}$$
yielding $\wdt K = \gamma Bg(\wdt w,\wdt y) + p B\id(\wdt w, \wdt y) - \wdt F B\xi(\wdt w,\wdt y)$.  Note that $\wdt K > 0$ since $\wdt F < h(\wdh y)$.  It then follows from \eqref{e:F_K} that for this  fixed cost $\wdt K$, the $(\wdt w,\wdt y)$-impulse policy  has a long-term average reward  $\wdt F = h(\wdt y) > h(\wdh y) - \eps$.  
 Moreover, \remref{rem-comparison-singular} shows that the difference $h(\wdh y) - h(\wdt y)$ includes the fixed cost rate $\wdt K/B\xi(\wdt w,\wdt y)$ when the fixed cost $\wdt K$ is charged for each intervention.
  For the singular control problem, however, there is no fixed cost so the $(\wdt w,\wdt y)$-policy will not incur this charge, increasing the value.  Hence the corresponding value $\wdh J(R^{(\wdt w,\wdt y)})$ also satisfies $\wdh J(R^{(\wdt w,\wdt y)}) > h(\wdh y) - \eps$.  The ``chattering'' policy $R^{(\wdt w,\wdt y)}$ is an implementable $\eps$-optimal policy for the singular control problem.
\end{rem}

\section{Conclusions and Remarks}\label{sect-conclusion} 
This paper presents a comprehensive analysis of a class of long-term average impulse control problems  having two sources of income, together with their closely related singular control formulations. Working within a general one-dimensional diffusion framework with boundary conditions motivated by applications, we derive explicit solutions and characterize the optimal controls for both problems.  Our findings rigorously establish an intrinsic connection between the two control formulations, with the singular control solution emerging as a natural limit of the impulse control problem as fixed costs vanish.   It also identifies the penalty incurred on the optimal impulse value (relative to the optimal singular value) by the presence of the positive fixed cost $K$.  This penalty includes a reduction due to the restriction to the subclass of impulse controls along with the long-term average reduction arising directly from the fixed costs.  

 The paper conducts a detailed sensitivity analysis of the impulse control solution with respect to key parameters, providing practical guidance for applications, for example by identifying parameters that require particularly accurate estimation. The related long-term total reward criterion and its connection to overtaking optimality are also explored.  The explicit nature of the solutions makes the results directly applicable to canonical problems in areas such as renewable resource  and portfolio management, where decision makers must balance continuous subsidies against the income from interventions. 

The explicit solutions obtained here provide a solid foundation for the study of long-term average impulse control with mean field interactions in the companion paper \cite{HelmSZ:25}.

Several avenues for future research remain. Some of the technical conditions and assumptions imposed in this paper may be relaxed. Moreover, the framework could be extended to more general state processes, including those with McKean-Vlasov dynamics, to risk-sensitive objective functions, or to an exploration of the connections between long-term average, discounted and finite-time horizon problems, homogenization theory, and turnpike properties in the context of impulse and singular controls.

 \appendix
 \section{Technical Proofs}
 \label{Appen-proofs-sect2}
 
  \subsection{Proofs of Section  \ref{sect-formulation}}  \label{appen-pfs-sec2}
 \begin{proof}[Proof of Proposition \ref{lem-mu-at-b}] 
(i) Since $a$ is a non-attracting point by Condition \ref{diff-cnd}, $S(a, y] = \infty$ for any $y\in \I$; see  Table 6.2 on p.~234  of \cite{KarlinT81}. This, in turn, implies that $s(a+) =\infty $. Since $  \sigma^{2}(x) > 0$ for  $x\in \I$, the fact that $s(a+) =\infty $ also implies  that $\mu(a) \ge 0$.  Moreover, we can apply L'H\^opital's rule to derive 
\begin{equation}\label{e:xi':a}
\lim_{x\downarrow a} s(x) M(a, x] = \lim_{x\downarrow a} \frac{M(a, x] }{\frac{1}{s(x)}} = \lim_{x\downarrow a} \frac{m(x)}{\mu(x) \frac{2}{s(x) \sigma^{2}(x)}} = \frac1{\mu(a)}.
\end{equation} In particular, if $\mu(a) > 0$, then   for any $y \in \I$, we have \begin{displaymath}
  \int_{a}^{y} s(x) M(a,x] dx =  \int_{a}^{y}   M(a,x] dS(x) < \infty.
\end{displaymath} This together with $S(a, y] = \infty$ imply  that $a$ is an entrance  boundary in view of Table 6.2 on p.~234 of \cite{KarlinT81}. 

(ii) Since  $\frac{d}{du} \big(\frac1{s(u)} \big)  = \frac{2\mu(u)}{ s(u) \sigma^{2}(u)} = \mu(u) m(u),$     we have 
\begin{displaymath}
\int_{l}^{x} \mu(u) m(u)\, du =  \int_{l}^{x} \mu(u)\,  dM (u) = \frac1{s(x)}- \frac1{s(l)}, \quad \forall a <l < x<b.
\end{displaymath} 
Therefore  letting $l \downarrow a$, the bounded convergence theorem gives \eqref{e:1/s-identity}.

Turning now to the proofs of (iii) and (iv), we first note that  for each $x \in \I$, dividing \eqref{e:1/s-identity} by $M[a,x]$ and using \eqref{e-sM-infty} yields
\begin{equation} \label{limit=0}
\ell(x) := \frac1{M[a,x]}\int_a^x \mu(u)\, {dM(u)} = \frac{1}{s(x) M[a,x]} \stackrel{x\to b}{\longrightarrow} 0.
\end{equation} In view of the expression for $\xi' $ in \eqref{e:xi-derivatives}, we have $\ell(x) = \frac1{\xi'(x)}$ for any $x\in \I$.
We show that $\ell(b)$ takes value $\langle \mu,\pi \rangle$ in (iii) and $\mu(b)$ in (iv).

  (iii)   Consider the case   when $M[a,b] < \infty$.   Since $\mu$ is continuous on $\I$ and extends continuously to the boundary points by Condition \ref{diff-cnd}(c),  $\mu(b)$ can be either $ -\infty$, $\infty$, or finite. Let us first consider the case when $\mu(b) = -\infty$.  
  Then there exists a $y_{0} \in \I$ so that $\mu(y) < 0$ and $M[a, y] > \frac12 M[a, b]$ for all $y\in (y_{0} , b)$. 
  Recall we have observed that $\mu(a) \ge 0$ and since $\mu$ is continuous, $\mu$ is bounded on $[a, y_{0}]$. 
  Now for any $x\in [y_{0}, b)$, we can write 
\begin{align} \label{e:mu-dM} 
\ell(x) & = \frac{M[a, b]}{M[a, x]} \left(\int_{a}^{y_{0}} \mu(y) d\pi(y) + \int_{y_{0}}^{b}  \mu(y)I_{\{y\in [y_{0} ,x] \}} d\pi(y) \right).
\end{align} 
We have    $\lim_{x\to b}\frac{M[a,b]}{M[a,x]} = 1$ and, thanks to the Monotone Convergence Theorem, $$\lim_{x\to b}  \int_{y_{0}}^{b}  \mu(y)I_{\{y\in [y_{0} ,x] \}} d\pi(y) =  \int_{y_{0}}^{b}  \mu(y)  d\pi(y) .$$ Note also that $\int_{a}^{y_{0}} \mu(y) d\pi(y) $ is bounded.  
Using these observations in \eqref{e:mu-dM}, we have  
\begin{displaymath}
\lim_{x\to b} \ell(x) =  \int_{a}^{y_{K}} \mu(y)\, d\pi(y) + \int_{y_{K}}^{b}  \mu(y)\, d\pi(y) = \int_{a}^{b}  \mu(y)\, d\pi(y)= \lan\mu,\pi\ran,
\end{displaymath} 
which, together with \eqref{limit=0}, establish   assertion (iii) of the proposition  when $\mu(b) = -\infty$. The proof for the case when  $\mu(b) = \infty$ is similar. 

When $\mu(b)$ is finite, $\mu$ is bounded on $[a,b]$ and $\lim_{x \to b} \mu(y) I_{\{y\in [a,x]\}} = \mu(y)$ for all $y \in [a,b]$.  The bounded convergence theorem then implies
$$0= \lim_{x\to b} \ell(x) =\lim_{x\to b} \frac{M[a, b]}{M[a, x]} \int_a^b \mu(y) I_{\{y \in [a,x]\}}\, \pi(dy) = \int_a^b \mu(y)\,\pi(dy).$$  
(Defining the probability measures $\pi_x$ for $x \in (a,b)$ by $\pi_x(du) = I_{\{u \in [a,x]\}}\, \frac{dM(u)}{M[a,x]}$, this argument for an arbitrary bounded, continuous function $f$ on $[a,b]$ also proves that $\pi_x \Rightarrow \pi$ and, as we will use in the proof of (iv), $\pi_x \Rightarrow \delta_b$.)

 (iv) If $M[a, b] = \infty$, we verify that $\mu(b)$ is finite using proof by contradiction.  
First, observe that if $\mu(u) \to \infty$ as $u \to b$, then for any $K>0$, there is some $z_K$ such that for all $u \in [z_K,b)$, $\mu(u) \geq K$.  Recall $\mu(a) \geq 0$ and $\mu$ is continuous, so $|\mu| \leq K_1$ on $[a,z_K]$ for some $K_1 < \infty$.  Since $M[a,x] \to \infty$ as $x \to b$,
\begin{align*}
\lim_{x\to b} \ell(x) =\liminf_{x\to b} \int_a^x\mu(u)\, \frac{dM(u)}{M[a, x]} & = \liminf_{x\to b} \left( \int_a^{z_K} \mu(u)\, \frac{dM(y)}{M[a,x]} + \int_{z_K}^x \mu(u)\, \frac{dM(u)}{M[a,x]} \right) \\
& \geq \liminf_{x\to b} \left( \frac{-K_1\, M[a,z_K]}{M[a, x]} + \frac{K\, M[z_K, x]}{M[a,x]} \right) = K,
\end{align*}
which contradicts \eqref{limit=0}.  Thus $\mu(b)\neq \infty$.  A similar argument shows $\mu(b) \neq -\infty$ and therefore $\mu(b)$ is finite.  

Again for each $x \in \I$, define the probability measure $\pi_x$ on $[a,b]$ such that $\pi_x(G) = \int_G I_{[a,x]}(y)\, \frac{dM(y)}{M[a,x]}$ for $G \in \B[a,b]$.  Notice that $\lim_{x\to b} M[a,x] = \infty$ implies $\pi_x \Rightarrow \delta_b$ as $x \to b$.  Since $\mu$ is bounded and extends continuously on $[a,b]$ by assumption, it follows from weak convergence and \eqref{limit=0} that $ \mu(b) = \lim_{x \to b} \ell(x) = 0$.
\end{proof} 
 
\subsection{Proofs of Section \ref{sect-preliminaries}}\label{appen-pfs-sec3}
\begin{proof}[Proof of Lemma \ref{lem-limit:Bid/Bxi-a}]
(i) Since $a > -\infty$ is natural, 
  it follows from Table 6.2 on p.~234 of \cite{KarlinT81} that  \begin{equation}
\label{e-xi a=infty}
\lim_{w\to a} (\xi(y_{0}) - \xi(w)) = \int_{a}^{y_{0}} M(a, v] dS(v)   =  \infty, \ \ \forall y_{0}\in \I.
\end{equation} 
 Then  the first two limits in \eqref{e1:limit:Bid/Bxi-a} follow directly from   \eqref{e-xi a=infty}. 
   We now prove  the third limit. Since $a$ is non-attracting and natural, we can apply Proposition \ref{lem-mu-at-b}(i) and \eqref{e:xi':a} to derive $$ \lim_{x\to a} \frac{1}{\xi'(x)} =  \lim_{x\to a} \frac{1}{s(x) M(a, x]} = \mu(a) = 0.$$ 
   On the other hand,  the mean value theorem says that   $  \frac{y  -w}{ \xi(y ) - \xi(w)} = \frac{1}{\xi'(\theta)} 
$ for some $\theta$  between $w$ and $y$. Therefore, if $(w, y)\to (a, a)$, then $\theta \to a$. These observations  lead to the third limit of  \eqref{e1:limit:Bid/Bxi-a}.

(ii) The first two limits in \eqref{e:limit:Bid/Bxi-b} follow directly from \eqref{eq:xi_limit_b} if $b < \infty$, and  
 from L'H\^opital's Rule and \eqref{e-sM-infty} if $b = \infty$.  The last limit  in \eqref{e:limit:Bid/Bxi-b}  follows from the mean value theorem and \eqref{e-sM-infty}. 
\end{proof}

\begin{proof}[Proof of Lemma \ref{lem1-Bg/Bxi:a}]
For any $\e > 0$, thanks to the continuity of $c$, we can pick a $w_{\e} > a$ so that \begin{displaymath}
|c(w) - c(a)| < \e, \quad \forall w\in (a, w_{\e}].  
\end{displaymath} It then follows that for any $(w, y) \in \cR $ with $y\le w_{\e}$, we have \begin{equation}\label{e:Bg} \begin{aligned}
 (c(a) - \e)(\xi(y) - \xi(w)) & < g(y) - g(w)\\& = \int_{w}^{y} \int_{a}^{v} c(u) dM(u) dS(v) <    (c(a) + \e)(\xi(y) - \xi(w)). \end{aligned}
\end{equation} This gives  $ \lim_{(w,y)\to (a,a)} \frac{g(y) - g(w) }{\xi(y) - \xi(w)}= c(a)$ and hence establishes \eqref{e0:g/xi-limit:a}. 

Next we prove \eqref{e:g/xi-limit:a} when $a$ is a natural point. To this end, fix an  arbitrary  $y_{0}\in \I$.  
Conditions \ref{diff-cnd}(b) and \ref{c-cond} imply that $\int_{a}^{v} c(u) dM(u)$ is uniformly bounded on $v\in [w_{\e} \wedge y_0, w_{\e}\vee y_0] $
 This, in turn, implies that  $g(y_{0}) - g(w_{\e})$ is bounded. Similarly, the condition $M[a, y] < \infty$ for each $y \in \I$ implies that  $\xi(y_{0}) - \xi(w_{\e})$ is finite. Using these observations as well as \eqref{e-xi a=infty} and \eqref{e:Bg},  we can compute  
\begin{align*} 
 \limsup_{w\to a}\frac{g(y_{0}) - g(w) }{\xi(y_{0}) - \xi(w)} & =\limsup_{w\to a} \frac{g(y_{0}) - g(w_{\e}) +g(w_{\e}) - g(w)  }{\xi(y_{0}) - \xi(w_{\e}) +\xi(w_{\e}) - \xi(w)} \\
 & =  \limsup_{w\to a} \left(\frac{g(w_{\e}) - g(w) }{\xi(w_{\e}) - \xi(w)} + \frac{g(y_{0}) - g(w_{\e})}{\xi(w_{\e}) - \xi(w)}\right) \cdot \frac1{1+ \frac{\xi(y_{0}) - \xi(w_{\e})}{\xi(w_{\e}) - \xi(w)}}\\
 & \le  c(a) +\e,  
\end{align*}  and similarly \begin{displaymath}
 \liminf_{w\to a}\frac{g(y_{0}) - g(w) }{\xi(y_{0}) - \xi(w)} \ge c(a) -\e. 
\end{displaymath} Hence the limit $ \lim_{w\to a} \frac{g(y_{0}) - g(w) }{\xi(y_{0}) - \xi(w)} = c(a)$ holds. Finally, by talking  $y_{0} = x_{0}$, we have  $\lim_{w\to a} \frac{g(w) }{ \xi(w)} = c(a)$. The proof is complete. 
\end{proof}

\begin{proof}[Proof of Lemma \ref{lem2-Bg/Bxi:b}] 
Arbitrarily fix some $y_0 \in \I$. Since $c$ is   increasing with $c(b) > c(a) \ge 0$, we can find some $y_{1} > y_{0} $ so that $c(x) \ge c(y_{1}) > 0$ for all $x\in [y_{1}, b)$. Then for all $y\ge y_{1}$, 
\begin{align*} 
g(y) - g(y_{0}) & \ge g(y) - g(y_{1})  = \int_{y_{1}}^{y} \int_{a}^{v} c(x)dM(x) dS(v)  \\
 & \ge \int_{y_{1}}^{y}  \int_{y_{1}}^{v} c(x)dM(x) dS(v) \ge c(y_{1}) \int_{y_{1}}^{y} M[y_{1}, v] dS(v); 
\end{align*} the last expression converges to $\infty$ as $y\to b$ thanks to \eqref{eq:xi_limit_b}. Therefore we can apply  L'H\^opital's rule  to obtain \begin{equation}\label{e1-lem35-pf}
\lim_{y\to b} \frac{g(y) - g(y_{0})}{\xi(y) - \xi(y_{0})} = \lim_{y\to b}    \frac{s(y)  \int_{a}^{y}  c(u) dM(u)}{s(y) M[a, y]} =  \lim_{y\to b} \frac{ \int_{a}^{y}  c(u) dM(u)}{ M[a, y]}.
\end{equation} If $M[a, b] =\infty$, again using the fact that $c(x) \ge c(y_{1}) > 0$ for all $x\in [y_{1}, b)$, we have $\lim_{y\to b}  \int_{a}^{y}  c(u) dM(u) = \infty$. Therefore another application of L'H\^opital's rule leads to \begin{equation}\label{e1.1-lem35-pf}
  \lim_{y\to b} \frac{ \int_{a}^{y}  c(u) dM(u)}{ M[a, y]} =  \lim_{y\to b} \frac{c(y) m (y)}{m(y)} = c(b). 
\end{equation}  If $M[a, b] <\infty,  $   we can choose a $y_{2} \in (y_1,   b)$ so that $1\le \frac{M[a, b]}{M[a, y]} \le 2$ for all $y\ge y_{2}$. Then for $y\ge y_{2}$, we have   \begin{align}\label{e1.2-lem35-pf} \nonumber
  \frac{ \int_{a}^{y}  c(u) dM(u)}{ M[a, y]}&  = \int_{a}^{b}c(u) \frac {M[a, b]}{M[a, y]}I_{\{u\in [a, y]\}} \frac{ dM(u)}{M[a, b]}\\ & =  \int_{a}^{b}c(u) \frac {M[a, b]}{M[a, y]}I_{\{u \in [a, y]\}} \pi(du) \to  \int_{a}^{b} c(u) \pi(du) = \langle c, \pi\rangle ;
\end{align}   as $y\to b$, where we  used the dominated convergence theorem to derive the convergence in the second line of \eqref{e1.2-lem35-pf}. Plugging \eqref{e1.1-lem35-pf} and \eqref{e1.2-lem35-pf} into \eqref{e1-lem35-pf} gives us the limit \begin{equation}\label{e2-lem35-pf}
\lim_{y\to b} \frac{g(y) - g(y_{0})}{\xi(y) - \xi(y_{0})}  = \bar c(b) = \begin{cases}
 c(b)     & \text{ if  }M[a, b] =\infty, \\
   \langle c, \pi\rangle    & \text{ if } M[a, b] < \infty.
\end{cases}
\end{equation} Taking $y_{0} = x_{0}$ in \eqref{e2-lem35-pf} yields the first limit $\lim_{y\to b} \frac{g(y)}{\xi(y) } = \bar c(b)$ in  \eqref{e1:g/xi-limit:b}.  

It remains to show that $ \frac{g(y) - g(w) }{\xi(y) - \xi(w)}$ converges to the same limit $\bar c(b) $ as $(w, y) \to (b, b)$.   Recall that $\xi'(x)> 0$ for all $x\in \I$. 
 Hence we can apply the generalized mean value theorem 
  and the expressions for $\xi'$ and $g'$ in \eqref{e:xi-derivatives} and \eqref{e:g-derivatives} respectively to compute
\begin{align} \label{e:Bg/Bxi-bb}
 \frac{g(y) - g(w) }{\xi(y) - \xi(w)} & = \frac{g'(\theta)}{\xi'(\theta)}
  =\frac{ \int_{a}^{\theta} c(u) dM(u)}{M[a,  \theta]},
\end{align} 
where $\theta\in (w, y)$. In particular, we have $\theta\to b$ when     $(w,y) \to (b, b)$. As a result, using the same arguments as those for \eqref{e1.1-lem35-pf} and \eqref{e1.2-lem35-pf},  we see that the last expression of \eqref{e:Bg/Bxi-bb} converges to $\bar c(b)$ as $(w,y) \to (b, b)$. 
Therefore all limits in \eqref{e1:g/xi-limit:b} are  established. The proof is complete.   
\end{proof}

\subsection{Proof of Proposition \ref{prop-Fmax}}  \label{sect:Appen-B}

\begin{proof} 
We follow the approach in \cite{helm:18} to prove this proposition. 
Let $(\wdt w, \wdt y) \in \cR$ be as in Condition \ref{cond-interior-max}.    
We will show  that the value of $F$ on the boundary of $\cR$ is less than $F(\wdt w, \wdt y) $. Therefore it follows that the maximum value of $F$ is achieved at some point  $(w^*,y^*)\in \cR$. In case $a$ is an entrance boundary, it possible that the maximizing pair $(w^*,y^*)$ of $F$ satisfies $w^{*} =a$. The first order optimality condition then leads to    \eqref{e-1st-order-condition} or \eqref{e2-1st-order-condition}. 

We now spell out the detailed arguments. 
\begin{enumerate}
  \item[(i)] {\em The diagonal line $y=w$}.  
  
  Since $  K > 0$, we obviously   have $$\lim_{y-w\to 0, (w, y)\in \cR } F(w, y) =-\infty .$$ 
  \item[(ii)] {\em The boundary $y = b$, excluding $(a, b)$ when $a$ is natural}. 
  
  Fix an arbitrary  $w\in \I$. 
Thanks to Lemmas \ref{lem2-Bg/Bxi:b} and \ref{lem-limit:Bid/Bxi-a}, we have  
\begin{displaymath}
 \lim_{y\to b}  F(w, y) =    \lim_{y\to b} \frac{\gamma Bg(w, y)}{B\xi(w, y)} + \lim_{y\to b} \frac{p(y-w)- K}{\xi(y) - \xi(w)}  = \gamma \bar c(b). 
\end{displaymath} 

  \item[(iii)] {\em The vertex $(b, b)$}.  
  
Note that \begin{align*} 
\limsup_{(w,y) \to ( b, b)} \frac{ p(y-w) - K}{\xi(y) - \xi(w)}  \le \limsup_{(w,y) \to ( b, b)} \frac{p(y-w)}{\xi(y) - \xi(w)} =0,
\end{align*} 
where  the last equality follows from \eqref{e:limit:Bid/Bxi-b}.  Consequently, as in the previous case, this and Lemma~\ref{lem2-Bg/Bxi:b} imply that 
\begin{displaymath}
 \limsup_{(w,y) \to ( b, b)}  F(w, y) \le \gamma \bar c(b). 
\end{displaymath}

\item[(iv)] {\em The left boundary $w=a$, including the vertex $(a, a)$ when $a$ is a natural boundary}.

First we note that 
$$\limsup_{(w,y) \to (a, a),(w,y)\in \cR}  \frac{p(y-w)- K}{\xi(y) - \xi(w)}  \le 0. $$
 Also recall the limit \eqref{e0:g/xi-limit:a}.   Thus we have
\begin{displaymath}
\limsup_{(w,y) \to (a, a),(w,y)\in \cR} F(w, y) \le \gamma c(a). 
\end{displaymath}

Next, we fix some $y\in \I$ and assume $a$ is natural.  Then it follows from   Lemmas \ref{lem1-Bg/Bxi:a} and \ref{lem-limit:Bid/Bxi-a} that  
\begin{align*} 
\limsup_{w \downarrow a } F(w, y)& \le \limsup_{w \downarrow a }  \frac{\gamma(g(y) - g(w))}{\xi(y) - \xi(w)} + \limsup_{w \downarrow a }  \frac{p(y-w)- K}{ \xi(y) - \xi(w) } = \gamma c(a).  
\end{align*}
 When $a$ is an entrance boundary, $\xi(a): = \lim_{w\downarrow a} \xi(w) > -\infty$ thanks to Table 6.2 of \cite{KarlinT81}. On the other hand, for any $y\in \I$ fixed, since $c \ge 0$, the monotone convergence theorem implies that $\lim_{w \downarrow a} [g(y) - g(w)] = g(y)- g(a)$. Note also that   $a > -\infty$. Therefore  we have $$\lim_{w \downarrow a } F(w, y)= \frac{\gamma(g(y) - g(a)) + p(y-a) -K}{\xi(y) -\xi(a)}.$$ 
 It is possible  that an optimizing pair $(w^{*}, y^{*})$ of $F$ satisfies $w^{*} =a$. This possibility is accounted for in \eqref{e-Fmax} since $a\in \E$ and so $(a, y^{*})\in \cR$.

  \item[(v)] {\em The vertex $(a, b)$ when $a$ is a natural boundary}.   
  
  We fix an arbitrary  positive number $2 \e < F(\wdt w, \wdt y) -  \gamma \bar c(b)  $. 
    Consider the case when $b =\infty$; the other cases when $b< \infty$ can be proved by similar  arguments. 
    For this chosen $\e$,  let $R: = \frac{3}{2 \e }$.  Thanks to  \eqref{e-sM-infty}, there exists a $y_{0} \in (a,  b)$ so that 
\begin{displaymath}
\xi(y) - \xi(y_{0})  = \xi'(\theta) (y-y_{0}) = s(\theta) M[a, \theta]  (y-y_{0}) \ge p R    (y-y_{0}), 
\end{displaymath} 
for all $y > y_{0}$, where $\theta$ is between $y_{0}$ and $y$. Consequently for all $y> y_{0} \vee \frac{a+ 1 + K}{p}$ and $a < w < (a+1)\wedge y_{0}$, we have 
\begin{align*} 
   \frac{p(y-w)-  K}{\xi(y) - \xi(w) }   &=  \frac{p(y-w)- K}{\xi(y) - \xi(y_{0}) + \xi(y_{0})-\xi(w)} \le \frac{p(y-a)- K}{p R   (y-y_{0})} .\end{align*}
Furthermore, since $\lim_{y\to\infty} \frac{p(y-a)-  K}{pR(y-y_{0})} =\frac1R$, we can find a $\wdt y_{0} >  y_{0} \vee \frac{a+ 1 +  K}{p} $ so that   \begin{equation}\label{e1:region-v} \frac{p(y-w)-   K}{\xi(y) - \xi(w) }   \le \frac{p(y-a)-  K}{pR(y-y_{0})} < \frac{3}{2 R } = \e,  
 \end{equation}  for all $(w, y) \in \cR $ with $w \in (a, (a+ 1)\wedge y_{0})$ and $y > \wdt y_{0}$.   
 
 Using Lemmas \ref{lem1-Bg/Bxi:a} and \ref{lem2-Bg/Bxi:b}, we can find $\wdh w_{0} \in (a, (a+ 1)\wedge y_{0})$ and $\wdh y_{0}\in (\wdt y_{0}, b)$ so that \begin{align*} 
\gamma (g(y) - g(x_{0}))  & \le  
  (\gamma \bar c(b) + \e) (\xi (y) - \xi(x_{0})),     \ \quad  \text{ for all } y \ge \wdh y_{0}, \\
 \gamma(g(x_{0}) - g(w)) & \le (\gamma c(a) + \e) (\xi(x_{0}) - \xi(w) ),  \quad \text{ for all } w \le \wdh w_{0}. 
\end{align*} 
Consequently for all $(w, y)\in \cR$ satisfying $w \le \wdh w_{0}$ and $y \ge \wdh y_{0}$, we have 
 \begin{align} \label{e2:region-v}
\frac{\gamma(g(y) - g(w))}{\xi(y) - \xi(w)}  & = \frac{\gamma(g(y) - g(x_{0}) )+   \gamma(g(x_{0})- g(w))}{\xi(y) -\xi(x_{0}) + \xi(x_{0})-  \xi(w)}  \le \gamma\bar c(b)\vee \gamma c(a) + \e =\gamma \bar c(b) +\e.
\end{align} 
A combination of \eqref{e1:region-v} and \eqref{e2:region-v} then gives us 
\begin{displaymath}
F(w, y) = \frac{\gamma(g(y) - g(w))}{\xi(y) - \xi(w)} +  \frac{p(y-w)-  K}{\xi(y) - \xi(w) }  < \gamma \bar c(b) + 2\e < F (\wdt w, \wdt y) 
\end{displaymath} 
for all $(w, y)\in \cR$ with $w \le \wdh w_{0}$ and $y \ge \wdh y_{0}$.
\end{enumerate}
Finally combining all the above cases, we see that the maximum value of $F$ must occur at some point $(w^{*}, y^{*} ) \in \cR$; this establishes  \eqref{e-Fmax}. 
If $(w^{*}, y^{*})$ is an interior extreme point, we have $\frac{\partial F(w^{*}, y^{*}}{\partial w}) = \frac{\partial F(w^{*}, y^{*})}{\partial y} =0$; these equations lead to \eqref{e-1st-order-condition}. 
 If $a$ is an entrance point and  $(w^{*}, y^{*}) = (a, y^{*})$, then we have $\frac{\partial F(w^{*}, y^{*})}{\partial w} \le 0 = \frac{\partial F(w^{*}, y^{*})}{\partial y} $; these equations lead to \eqref{e2-1st-order-condition}.
\end{proof} 

\def\cprime{$'$} \def\polhk#1{\setbox0=\hbox{#1}{\ooalign{\hidewidth
  \lower1.5ex\hbox{`}\hidewidth\crcr\unhbox0}}}

\end{document}